\documentclass
[11pt, a4paper]{amsart}
\usepackage[utf8]{inputenc}
\usepackage{amssymb,amscd}
\usepackage{thmtools}
\usepackage{algorithm}
\usepackage{algpseudocode}

\usepackage[textsize=tiny]{todonotes}
\usepackage{hyperref}
\usepackage[shortlabels]{enumitem}
\usepackage{float} 
\usepackage{cleveref}   

\newcommand{\Kbar}{\overline{\mathbb{K}}}


\newtheoremstyle{alstandard}{7pt}{3pt}{\rm}{}{\scshape}{:}{0.5em}{}
\theoremstyle{alstandard}
\swapnumbers
\declaretheorem[name=Theorem]{theorem}
\numberwithin{theorem}{section}
\declaretheorem[sibling=theorem, name=Lemma]{lemma}
\declaretheorem[sibling=theorem, name=Proposition]{prop}

\declaretheorem[sibling=theorem, name=Definition]{defi}

\declaretheorem[sibling=theorem, name=Corollary]{cor}
\declaretheorem[sibling=theorem, name=Remark]{remark}

\newcommand{\Q}{\mathbb{Q}}
\newcommand{\R}{\mathbb{R}}
\newcommand{\C}{\mathbb{C}}
\newcommand{\N}{\mathbb{N}}

\renewcommand\P{\mathbb P}

\newcommand{\K}{\mathbb K}

\DeclareMathOperator{\im}{im}

\DeclareMathOperator{\rk}{rank}

\DeclareMathOperator{\codim}{codim}

\DeclareMathOperator{\grass}{Gr}

\DeclareMathOperator{\skrk}{skrank}

\usepackage{pifont}
\begin{document}
	\title[Border rank=rank for Kruskal tensors]{Border rank=rank for Kruskal tensors and a Kruskal's theorem for skew decompositions}

    \author[Blomenhofer]{Alexander Taveira Blomenhofer\textsuperscript{*}}
    \thanks{*\:\: University of Copenhagen}

	
    \author[Lovitz]{Benjamin Lovitz\textsuperscript{**}}
	\thanks{** Concordia University, Montréal}
	\begin{abstract} 
		We show that border rank is equal to rank for Kruskal tensors. We also give an analogous Kruskal condition for alternating tensors, which certifies uniqueness of skew rank decompositions. Furthermore, we show that border skew rank is equal to skew rank for alternating Kruskal tensors, and we give an algorithm to find the minimum skew rank decomposition of alternating Kruskal tensors.
	\end{abstract}
	
	\maketitle
	\raggedbottom
    
	\section{Introduction}
Rank, border rank and identifiability are three of the most difficult problems in tensor geometry. In particular, tensor rank is NP-hard to compute \cite{Hillar_Lim_2013}. 
These three problems share a similar characteristic that makes them challenging: 
Writing down one decomposition might be easy, but ruling out other decompositions can be very difficult. If a tensor $ T\in \mathbb K^{n_1\times n_2 \times n_3} $ (where $\mathbb K$ is any field, such as $\mathbb R$ or $\mathbb C$) has a known decomposition 
\begin{align}\label{eq:intro_decomp}
	T = \sum_{i = 1}^{r} a_i\otimes b_i\otimes c_i, \qquad (a_i\in \mathbb K^{n_1}, b_i\in \mathbb K^{n_2}, c_i\in \mathbb K^{n_3}),
\end{align}
then its rank is at most $ r $. However, to show that it has rank precisely $ r $, one needs to rule out  ``surprise decompositions'' of length \emph{less} than $ r $. 
To then show that $ T $ is identifiable, one needs to rule out other decompositions of length \emph{equal} to $ r $. Finally, to show that $ T $ also has border rank $ r $, over $ \K=\C $ one needs to rule out that $ T $ could be approximated by a curve of tensors $ T_{\varepsilon} \to T $ (as $ \varepsilon\to 0 $), where the $ T_{\varepsilon} $ have rank smaller than $ r $ for $ \varepsilon\neq 0 $. In other words, one needs to rule out surprise decompositions that only work in the \emph{limit}. 

Kruskal's theorem \cite{Kruskal_1977} is a classical result on tensor decomposition. It provides one of very few explicit, sufficient criteria to see that a given decomposition is the \emph{unique minimum rank decomposition}. 
In other words, it establishes both the rank and the identifiability of a given decomposition. 
Kruskal's theorem assumes a matroidal constraint on the decomposition, which is formulated via Kruskal ranks. These notions will be introduced in detail in \Cref{sec:prelims}. 
However, Kruskal's theorem makes no claim on border rank. 

\subsection{Main results} 
In this work, we settle the border rank question for Kruskal tensors. For a tensor decomposition~\eqref{eq:intro_decomp}, the \textit{Kruskal rank} of the $a$-vectors is defined as the largest integer $k_a$ such that every size $k_a$ subset of $\{a_1,\dots, a_r\}$ is linearly independent. The Kruskal ranks $k_b$ and $k_c$ are defined similarly. Kruskal's theorem states that if $2r \leq k_a+k_b+k_c-2$, then the decomposition~\eqref{eq:intro_decomp} is the unique tensor rank decomposition of $T$. Our first main result is that the border rank equals the rank in this case, even if the inequality is weakened by one:

\begin{theorem}[Theorem~\ref{thm:Kruskal_border_rank_precise}]
	Let $T \in \mathbb K^{n_1 \times n_2 \times n_3}$ be as in~\eqref{eq:intro_decomp}. If $2r \leq k_a + k_b + k_c -1$, then the border rank and rank of $T$ are both equal to $r$.
\end{theorem}

As a second contribution, we prove an analogous Kruskal-type theorem for alternating tensors. For an alternating tensor $T \in \Lambda^3(\mathbb K^n)$, a \emph{skew decomposition} is a decomposition of the form
\begin{align}\label{eq:intro_adecomp}
	T=\sum_{i=1}^r a_{i} \wedge b_{i} \wedge c_{i}.
\end{align}
We define the \textit{Kruskal rank} of the decomposition to be the largest integer $k$ for which every collection of $k$ of the three-dimensional subspaces $U_i:=\langle a_i, b_i, c_i \rangle \subseteq \mathbb K^n$ is in direct sum. Our second main result is the following:
\begin{theorem}[informal]
	Let $T \in \Lambda^3(\mathbb K^n)$ be as in~\eqref{eq:intro_adecomp}.
	\begin{enumerate}
		\item If $2r \leq 3k-2$, then~\eqref{eq:intro_adecomp} is the unique skew rank decomposition of $T$.
		\item The inequality $2r \leq 3k-2$ is sharp, analogously to Kruskal's theorem (see~\cite{derksen2013kruskal}).
		\item If $2r \leq 3k-1$, then the border skew rank and skew rank of $T$ are both equal to $r$.
		\item If $2r \leq 3k-2$, then the unique skew rank decomposition can be recovered from $T$ in time $n^{\mathcal{O}(r-k+1)}$.
	\end{enumerate}
\end{theorem}
These four statements are proven formally in Theorems~\ref{thm:alternating_kruskal},~\ref{thm:alternating_sharp},~\ref{thm:border-skew-rank=skew-rank} and~\ref{thm:algorithmic-alternating-kruskal}, respectively. Analogous algorithmic results for Kruskal's original theorem are known~\cite{domanov2014canonical,BSS26}; see the related work section for more details.

\subsection{Techniques}

Our techniques are of independent interest. We associate \emph{contraction varieties} $ X_m(T) $ with a tensor $ T $. Those are the varieties of contractions of a 3-tensor for which the resulting matrix has a bounded rank. Contraction varieties are classically studied invariants of tensors. They are a useful tool with a wide range of applications, which the authors recently used in an algorithm for Chow decompositions~\cite{Blomenhofer_Lovitz_2025}. We refer to \cite{Blomenhofer_Lovitz_2025} for further related work on contraction varieties. We will formally introduce contraction varieties in \Cref{sec:prelims}. 

As a further application of contraction varieties, we use them to give a more geometric proof of Kruskal's theorem in \Cref{sec:standard-kruskal}. Previous proofs of Kruskal's theorem relied on Kruskal's permutation lemma~\cite{Kruskal_1977,gubkin2024unique} or the splitting theorem~\cite{Lovitz_Petrov_2023}. These proofs are highly combinatorial in nature. Our proof is more geometric, in the sense that it establishes the uniqueness result mainly by computing the irreducible components of the contraction varieties of Kruskal tensors, with only minor combinatorics.

There are two main advantages of the more geometric approach. First, it allows us to also make statements about border rank, which is done in \Cref{sec:border-rank=rank}. Second, it generalizes easily to other settings, particularly to skew rank in \Cref{sec:alternating-kruskal}. 

In addition, there are two smaller advantages. First, the geometric approach allows us to give an intrinsic definition of Kruskal tensors, which does not make use of the rank decomposition. In particular, the intrinsic definition can be used to show that the set of $ r $-Kruskal tensors is a Zariski open subset of the set of rank-$ r $ tensors. See \Cref{sec:algorithms-and-examples} and \Cref{cor:Kruskal-tensors-open}.  Second, the framework of contraction varieties sometimes allows to obtain algorithms. As we remarked in \cite{Blomenhofer_Lovitz_2025}, whenever the contraction variety is a subspace arrangement of constant codimension, it can be computed efficiently. It turns out that this is sometimes the case for alternating Kruskal tensors, as we observe in~\Cref{sec:algorithmic_alternating}.

\subsection{Related work}
Kruskal's theorem \cite{Kruskal_1977} is one of the classical early uniqueness results for specific 3-tensors. It generalizes the uniqueness result attributed to Harshman and Jennrich~\cite{Harshman_1970} and remains a standard tool to establish the identifiability of a low-rank tensor. 
There exist several generalizations of Kruskal's theorem, notably~\cite{Lovitz_Petrov_2023,gubkin2024unique}. Reshaped Kruskal criteria exist for tensors of higher order. As Kruskal's original publication features a famously difficult proof, various attempts have also been made to find more digestible proofs of the theorem. 
Rhodes \cite{Rhodes_2010} gave a concise combinatorial proof of Kruskal's theorem. Landsberg \cite{Landsberg_2009} rephrased Kruskal's theorem in a more geometric language, but still relied on Kruskal's permutation lemma. Lovitz and Petrov gave a matroidal proof of Kruskal's theorem that does not rely on the permutation lemma~\cite{Lovitz_Petrov_2023}. Our proof in \Cref{sec:standard-kruskal} does not rely on Kruskal's permutation lemma, but instead uses primarily uniqueness of the decomposition of the contraction variety into irreducible components. 

A first algorithmic proof of Kruskal's theorem was given by Domanov and de Lathauwer \cite{domanov2014canonical}. The underlying proof technique is similar to our proof in \Cref{sec:standard-kruskal}. Our contribution here is a more geometric perspective, as we reprove Kruskal's theorem from the irreducible decomposition of the contraction variety. A recent algorithm for undercomplete skew-rank decomposition, allowing the rank to be up to $ r = \lfloor n/3\rfloor $,  was given in \cite{Vannieuwenhoven_2026_Chiseling}. 

While preparing this manuscript, we became aware of the recent work~\cite{BSS26} which also gives an algorithmic proof of Kruskal's theorem, similar to \cite{domanov2014canonical}, but with a thorough runtime analysis. The algorithmic uniqueness proof in \cite{BSS26} shares conceptual similarities both to the algorithm of Domanov and de Lathauwer \cite{domanov2014canonical} and to the geometric proof of (the standard) Kruskal's theorem presented in this work. 
Motivated by the approaches of \cite{domanov2014canonical,BSS26,Vannieuwenhoven_2026_Chiseling} and \cite{Blomenhofer_Lovitz_2025}, we developed an algorithmic version of our skew-symmetric Kruskal's theorem, presented in \Cref{sec:algorithmic_alternating}. In particular, we obtain an algorithm for skew-rank decomposition under an alternating Kruskal's condition whose runtime guarantees are comparable to those established in \cite{BSS26} for ``non-skew'' Kruskal tensors.

An open question of Bhaskara et al. asks whether there exists a \emph{polynomial-time} algorithmic formulation of Kruskal's theorem~\cite{bhaskara2014open}. Note that neither of the algorithms mentioned above runs in polynomial time. In fact, we show in \Cref{sec:algorithms-and-examples} that computing the dimension of contraction varieties of Kruskal tensors is NP hard in the worst case, making it unlikely to find efficient algorithms via this route.

\subsection{Model of computation} Throughout, we work
in the arithmetic circuit model of computation. Hence, we assume that addition, multiplication, and division over an arbitrary field can be conducted at unit cost. 
Furthermore, we assume access to an oracle that returns the roots of a univariate polynomial at unit cost per root. We also assume that we can  sample uniformly at random from a size-$k$ subset of $\K$ in time $\log k$.
We note that similar models of computation have been used in several tensor decomposition algorithms~\cite{johnston2023computing,koiran2025efficient,Kothari_Moitra_Wein_2025}.

	\section{Preliminaries}\label{sec:prelims}

\subsection{Notation}
We let $\mathbb K$ be a field with algebraic closure $\overline{\mathbb K}$, and consider tensors $ T\in \mathbb K^{n_1\times n_2\times n_3} $, of dimensions $ n_1, n_2, n_3\in \N $. In cases where the tensor dimensions are equal, we write $ n := n_1 = n_2 = n_3 $. The space of symmetric 3-tensors is denoted by $ S^3(\mathbb K^n) $ and the space of skew-symmetric 3-tensors is denoted by $ \Lambda^3(\mathbb K^n) $. 
Tensor decompositions are typically denoted as $ T = \sum_{i = 1}^{r} a_i\otimes b_i\otimes c_i $, where $ r\in \N_0 $ is the length of the decomposition (which upper-bounds the rank of the tensor) and $ a_i\in \mathbb K^{n_1}, b_i\in \mathbb K^{n_2}, c_i\in \mathbb K^{n_3} $. Formally, a tensor decomposition of $ T $ is a set of triples $ \{(a_1, b_1, c_1),\ldots,(a_r, b_r, c_r)\} $ such that $ T = \sum_{i = 1}^{r} a_i\otimes b_i\otimes c_i  $.

We fix a standard basis $e_1,\dots, e_n$ for $\overline{\mathbb{K}}^n$ and let $\langle \cdot, \cdot \rangle$ be the associated bilinear form. For a vector $f \in \Kbar^n$ we let $f^T=\langle f, \cdot \rangle \in (\Kbar^n)^*$.

\subsection{Contraction varieties}

\begin{defi}
	For a $ 3 $-tensor $ T\in \mathbb K^{n_1\times n_2\times n_3} $, we denote by $ T_f^{(i)} $ its contraction by the vector $ f\in \overline{\mathbb K}^{n_i} $ along the $i$-th index, where $i\in\{1,2,3\}$. Note that $ T_f^{(i)} $ is a matrix for each $f$ and $i$. E.g., $ T_f^{(3)} $ is a matrix of dimensions $n_1\times n_2$. Furthermore, we denote by $$ X_k^{(i)}(T) = \{f\in \Kbar^{n_i} \mid \rk T_f^{(i)}\le k\} $$ the \emph{rank-k contraction variety} of $ T $.

    In slight abuse of notation, we often write $ T_f $ and $ X_k(T)$ without indicating whether we mean the contraction along the first, second or third index. 
    It will be clear from the context which contraction variety is meant.
\end{defi}

\subsection{Kruskal tensors}

\begin{defi}
	Let $ \{a_1,\ldots,a_r\} $ be a multiset of $ r\ge 1 $ nonzero vectors in $ \mathbb K^{n} $. We define the Kruskal rank (short: k-rank) of $ \{a_1,\ldots,a_r\} $ to be the largest number $ \ell \in \{1,\ldots,r\} $ such that each subset of $ \{a_1,\ldots,a_r\} $ of cardinality $ \ell $ is linearly independent. 
\end{defi}

\begin{defi}
	Let $ \{a_1,\ldots,a_r\} \subseteq \mathbb K^{n_1} $ be a multiset of $ r\ge 1 $ nonzero vectors and likewise $ \{b_1,\ldots,b_r\} \subseteq \mathbb K^{n_2} $ and $ \{c_1,\ldots,c_r\} \subseteq \mathbb K^{n_3} $. Denote their k-ranks by $ k_a, k_b $ and $ k_c $. We say that the sets $ \{a_1,\ldots,a_r\}, \{b_1,\ldots,b_r\} $ and $ \{c_1,\ldots,c_r\} $ \emph{satisfy Kruskal's condition}, if $ 2r\le k_a + k_b + k_c - 2 $.
\end{defi}

\begin{defi}
	A \emph{Kruskal tensor} is a tensor $ T\in \mathbb K^{n_1\times n_2\times n_3} $ which has a decomposition $ T = \sum_{i = 1}^{r} a_i\otimes b_i\otimes c_i $ for which the multisets of vectors $\{a_i\}, \{b_i\},\{c_i\}$ satisfy {Kruskal's condition}, i.e. $ 2r \le k_a + k_b + k_c - 2  $.
\end{defi}

\noindent
We collect a few facts about Kruskal tensors, which can be found in \cite{Lovitz_Petrov_2023}.

\begin{prop}\label{prop:standard-kruskal}
	Let $a_1,\dots, a_r \in \mathbb K^{n_1}$ be vectors with Kruskal rank $k_a$, let $b_1,\dots, b_r \in \mathbb K^{n_2}$ be vectors with Kruskal rank $k_b$, and let $m$ be a non-negative integer for which $r+m \leq k_a+k_b-1$. If $\sum_{i=1}^r \alpha_i a_i b_i^T$ has rank $\leq m$ for some $\alpha=(\alpha_1,\dots, \alpha_r) \in \Kbar^r$, then $\omega(\alpha) \leq m$, where $\omega(\alpha)$ denotes the number of nonzero entries in $\alpha$.
\end{prop}
\begin{proof}
	Suppose toward contradiction that there exists $\alpha \in \Kbar^r$ such that $\sum_{i=1}^r \alpha_i a_i b_i^T$ has rank $\leq m$ but $\omega(\alpha)>m$. Assume without loss of generality that $k_a \leq k_b$. If $\omega(\alpha)\leq k_a$ then we have a contradiction, since the Kruskal rank assumption would then imply that the linear combination has rank $\omega(\alpha)$. If $k_a < \omega(\alpha) \leq k_b$, then
	\begin{align*}
		\rk(\sum_{i=1}^r \alpha_i a_i b_i^T) \geq k_a,
	\end{align*}
	a contradiction since by assumption, we have $k_a \geq r-k_b + m + 1 \ge m+1$. Suppose now that $\omega(\alpha)>k_b$, and note that $\omega(\alpha)-k_b \leq r-k_b \leq k_a-m-1$. Let $I \subseteq [r]$ be a subset of size $k_b$ for which $\alpha_i \neq 0$ for all $i \in I$. Then 
	\begin{align*}
		\rk(\sum_{i=1}^r \alpha_i a_i b_i^T) &\geq \rk(\sum_{i \in I} \alpha_i a_i b_i^T)-\rk(\sum_{i \notin I} \alpha_i a_i b_i^T)\\
		&\geq k_a-(k_a-m-1) = m+1.
	\end{align*}
	Here, in the second inequality we used $\rk(\sum_{i \in I} \alpha_i a_i b_i^T) \geq k_a$, which follows from the Kruskal rank assumption. This is a contradiction and completes the proof.
\end{proof}

\subsection{Alternating tensors}
In the following, we will also collect some generalizations of Kruskal's condition, which will be important for alternating tensors. Let us assume that $ \K $ is a field of characteristic zero or at least $ 5 $ (we make this assumption so that we can take the standard embedding $\Lambda^3(\mathbb K^n) \subseteq (\mathbb K^n)^{\otimes 3}$, although we suspect that our methods work also in smaller characteristic).
Each alternating tensor $ T\in \Lambda^3(\mathbb K^n) $ may be written as a sum of wedge products
\begin{align}\label{eq:skew-rank-decomp}
	T = \sum_{i = 1}^{r} a_i\wedge b_i\wedge c_i, 
\end{align}
where $x_1 \wedge x_2 \wedge x_3 := \frac{1}{3!} \sum_{\sigma \in \mathfrak{S}_3} \mathrm{sgn}(\sigma) x_{\sigma(1)} \otimes x_{\sigma(2)} \otimes x_{\sigma(3)}$. We assume that all terms in the decomposition are non-zero, so $ a_i, b_i, c_i\in \mathbb K^n $ span a 3-dimensional subspace $ U_i = \langle a_i, b_i, c_i \rangle $. If $ r $ is minimal, then \eqref{eq:skew-rank-decomp} is called a \emph{skew-rank decomposition} and $ r $ is called the \emph{skew rank} of $ T $, denoted $ \skrk(T) $. Some authors prefer to call $ \skrk(T) $ the \emph{alternating rank}. In the context of quantum information, $ \skrk(T) $ is also often called the \emph{Slater rank} \cite{Vrana_Levay_2008_Slater}. 

The main difference between skew rank and ``normal'' rank is that a skew-rank 1 term is associated with a 3-dimensional subspace $ \langle a_i, b_i, c_i \rangle $, whereas a tensor-rank 1 term is associated with three one-dimensional subspaces $ \langle a_i \rangle, \langle b_i \rangle $ and $ \langle c_i \rangle $. Therefore, it makes sense to consider a Kruskal condition more generally for subspaces rather than vectors. 

\begin{defi}
	Let $ \{U_1,\ldots,U_r\} $ be a multiset of $ r\ge 1 $ nonzero subspaces of $ \mathbb K^{n} $. We define the Kruskal rank (short: k-rank) of $ \{U_1,\ldots,U_r\} $ to be the largest number $ \ell $ such that each subset of $ \{U_1,\ldots,U_r\} $ of cardinality $ \ell $ is in direct sum. In other words, for all pairwise distinct choices of $ i_1,\ldots,i_{\ell} \in \{1,\ldots,r\} $, 
	\begin{align*}
		\dim (U_{i_1} + \ldots + U_{i_{\ell}}) = \dim U_{i_1} + \ldots + \dim U_{i_{\ell}}.
	\end{align*}
	The Kruskal rank of $ \{U_1,\ldots,U_r\} $ is denoted by $ k_U $. 
\end{defi}

Here, we denote by $ U_{i_1} + \ldots + U_{i_{\ell}} $ the subspace of $ \mathbb K^n $ generated by $ U_{i_1},\ldots,U_{i_{\ell}} $. Note that when the subspaces $U_i$ are all of the same dimension $d$, then $ k_U \le \frac{n}{d}$.

\begin{defi}
	We say that a collection $ \{U_1,\ldots,U_r\} $ of $ r\ge 1 $ nonzero subspaces of $ \mathbb K^{n} $ with Kruskal rank $ k_U $ satisfies the \emph{alternating Kruskal condition}, if 
	\begin{align*}
		2r \le 3k_U - 2. 
	\end{align*}
\end{defi}

\begin{defi}
	An \emph{alternating Kruskal tensor} is a tensor $ T\in \Lambda^3(\mathbb K^n) $ which has a decomposition $ T = \sum_{i = 1}^{r} a_i\wedge b_i\wedge c_i $ such that the spaces $ U_i := \langle a_i, b_i, c_i \rangle $ satisfy the alternating Kruskal condition. In other words,
	\begin{align*}
		2r \le 3k_U - 2.
	\end{align*}
\end{defi}

\subsection{Border rank}
The set of tensors $T\in \K^{n_1\times n_2\times n_3}$ of rank $r$ is in general not closed with respect to the Zariski topology. A similar problem arises for the set of alternating tensors $T\in \Lambda^3(\K^{n})$ of skew rank $r$. Introduced below, the notion of \emph{border rank} addresses this issue. Note that we say that an algebraic $\K$-variety $X$ in a $\K$-linear space $V$ is an \emph{affine cone} if it is closed under multiplication with scalars, so that $\K\cdot X \subseteq X$. We also assume in the following that the set $X$ spans $V$, i.e., $\langle X \rangle_{\K} = V$. This technical condition is to avoid elements whose border rank is infinite. 

\begin{defi}
    Let $X$ be an irreducible affine cone in a linear space $V$. 
    The rank-$s$ \emph{secant variety} is the Zariski closure of the set 
    \[
        \{x_1 + \ldots + x_s \mid x_1,\ldots, x_s \in X\}
    \]
    We denote the rank-$s$ secant variety by $\sigma_s(X)$. The \emph{border $X$-rank} of $T\in V$ is the smallest $s$ such that $T\in \sigma_s(X)$.
\end{defi}

\begin{remark}
   The two concrete examples of border rank studied in this paper are the following. 
    \begin{enumerate}
        \item Border tensor rank, often simply called border rank. Here, 
        \[
            X = \{a\otimes b\otimes c\mid a\in \K^{n_1},b\in \K^{n_2},c\in \K^{n_3}\}
        \]
        is the affine cone over the Segre variety of product tensors and $V =  \K^{n_1\times n_2\times n_3}$. 
        \item Border skew rank. Here, 
        \[
            X = \{a\wedge b\wedge c\mid a,b,c\in \K^{n}\}
        \]
        is the affine cone over the Grassmannian variety of decomposable alternating tensors and $V =  \Lambda^3(\K^{n})$. 
    \end{enumerate}
\end{remark}

\begin{remark}
    For $\K \in \{\R, \C\}$, the set of tensors $T\in \K^{n_1\times n_2\times n_3}$ of rank $r$ is generally also not closed with respect to the Euclidean topology.
    It is possible to ask for the border ranks with respect to the Euclidean topology as well. For $\K = \C$, this distinction makes no difference, since the Euclidean closure and the Zariski closure of the set of rank $r$ tensors in $\K^{n_1\times n_2\times n_3}$ agree. However, for $\K=\R$, one obtains different notions, which have been studied separately. See for instance the work \cite{Bernardi_Blekherman_Ottaviani_2018} on real typical ranks. 
\end{remark}

	\section{A more geometric proof of Kruskal's theorem}\label{sec:standard-kruskal}

In this section, we will give a proof of Kruskal's theorem based on contraction varieties. The proof of \Cref{thm:kruskal} combines two combinatorial (or matroidal) facts, \Cref{lem:combinatorics_general} and \Cref{thm:combinatorics_general}, with geometric reasoning.

\begin{lemma}\label{lem:combinatorics_general}
	Let $\mathbb K$ be a field, let $d\ge 1$, let
	$U_1,\ldots,U_r\subseteq \mathbb K^n$ be $d$-dimensional subspaces with
	Kruskal rank at least $k$, and let $\ell\in\{2,\ldots,k-1\}.$
	For each subset $I\subseteq [r]$ with $|I|=\ell$, write $U_I:=\langle U_i : i \in I \rangle$. Let $S$ be a collection of $r-\ell+1$ distinct
	$\ell$-element subsets of $[r]$ for which
	$\dim(\bigcap_{I\in S}U_I)=d\ell-d.$ Then there exists a subset $J\subseteq [r]$ with $|J|=\ell-1$ such that
	\[
		S=\{J\cup\{i\}:i\in [r]\setminus J\}.
	\]
\end{lemma}
\begin{proof}
	Let 
	$W:=\bigcap_{I\in S}U_I$ and $N:=|S|=r-\ell+1.$ By assumption, $\dim W=d(\ell-1).$ We first show that for any two distinct $I_1,I_2\in S$, it holds that
	\[
		U_{I_1}\cap U_{I_2}=W.
	\]
	Since $W\subseteq U_{I_1}\cap U_{I_2}$, it suffices to prove
	\[
		\dim(U_{I_1}\cap U_{I_2})\le d(\ell-1).
	\]
	Choose $a\in I_1\setminus I_2$. This is possible because
	$I_1\ne I_2$ and $|I_1|=|I_2|=\ell$. Since
	\[
		|I_2\cup\{a\}|=\ell+1\le k,
	\]
	the Kruskal-rank assumption implies that the subspaces indexed by
	$I_2\cup\{a\}$ are in direct sum. Hence $U_a\cap U_{I_2}=0$ and therefore $\dim(U_a+U_{I_2})=d(\ell+1).$ Since
	\[
		U_a+U_{I_2}\subseteq U_{I_1}+U_{I_2},
	\]
	we get
	\[
		\dim(U_{I_1}+U_{I_2})\ge d(\ell+1).
	\]
	Also, since $\ell\le k$, the subspaces indexed by $I_1$ and by $I_2$
	are each in direct sum, so $\dim U_{I_1}=\dim U_{I_2}=d\ell.$
	Thus,
	\begin{align*}
		\dim(U_{I_1}\cap U_{I_2})
		&=\dim U_{I_1}+\dim U_{I_2}-\dim(U_{I_1}+U_{I_2})\\
		&\le d\ell+d\ell-d(\ell+1)\\
		&=d(\ell-1).
	\end{align*}
	This proves the claim that $U_{I_1}\cap U_{I_2}=W.$
	
	Now define
	\[
		C:=\{i\in [r]: i \text{ belongs to at least two members of } S\}.
	\]
	We claim that $U_i\subseteq W$ for every $i\in C$. Indeed, if
	$i\in C$, then there are distinct $I_1,I_2\in S$ such that
	$i\in I_1\cap I_2$. Therefore
	\[
		U_i\subseteq U_{I_1}\cap U_{I_2}=W.
	\]
	Next we show that
	\[
		|C|\le \ell-1.
	\]
	If $|C|\ge \ell$, choose $C_0\subseteq C$ with $|C_0|=\ell$. Since $\ell\le k$, the subspaces $\{U_i:i\in C_0\}$ are in direct sum. Hence $\dim\left(\sum_{i\in C_0}U_i\right)=d\ell.$ But each $U_i$ with $i\in C_0$ is contained in $W$, so this
	$d\ell$-dimensional space is contained in $W$, contradicting $\dim W=d(\ell-1).$ Thus $|C|\le \ell-1.$
	
	We now count pairs $(i,I)$ with $i \in [r]$ and $i \in I \in  S$. Since each $I\in S$ has size $\ell$, the
	total number of such pairs is $N\ell.$ Every index in $C$ occurs in at most $N$ members of $ S$, while
	every index outside $C$ occurs in at most one member of $S$.
	Hence,
	\[
		N\ell\le |C|N+(r-|C|)=r+|C|(N-1).
	\]
	Using $|C|\le \ell-1$ and $N=r-\ell+1$, we obtain
	\[
		N\ell
		\le r+(\ell-1)(N-1)
		= r+(\ell-1)(r-\ell)
		= \ell(r-\ell+1)
		= N\ell.
	\]
	Hence equality holds throughout.
	
	In summary, $|C|=\ell-1,$ each index in $C$ belongs to all $N$ members of $ S$, and each
	index outside $C$ belongs to exactly one member of $ S$. Thus
	every member of $ S$ contains $C$. Since every member has size
	$\ell$, each member is of the form $C\cup\{i\}$ for a unique $i\in [r]\setminus C$. Because each index outside $C$
	appears exactly once, we get
	\[
		S=\{C\cup\{i\}:i\in [r]\setminus C\}.
	\]
	Taking $J=C$
	proves the claim.
\end{proof}

\begin{lemma}\label{thm:combinatorics_general}
	Let $\mathbb K$ be a field, and let $d\ge 1$ and $2\le k\le r$ be integers.
	Let $U_1,\ldots,U_r\subseteq \mathbb K^n$ and
	$U'_1,\ldots,U'_r\subseteq \mathbb K^n$ be two collections of
	$d$-dimensional subspaces, both of Kruskal rank at least $k$.
	For $I\subseteq [r]$, write $U_I=\langle U_i : i \in I \rangle$ and $U_I'=\langle U_i' : i \in I \rangle$. If
	\[
		\{U_I:|I|=k-1\}=\{U'_I:|I|=k-1\},
	\]
	then
	\[
		\{U_1,\ldots,U_r\}=\{U'_1,\ldots,U'_r\}.
	\]
\end{lemma}

\begin{proof}
	For $1\le t\le k-1$, set
	\[
		S_t:=\{U_I:|I|=t\},\qquad
		S'_t:=\{U'_I:|I|=t\}.
	\]
	If $k=2$, then the assumption already says $S_1=S'_1$, so there is
	nothing to prove. Hence assume $k\ge 3$.  We prove that $ S_t$ determines $ S_{t-1}$ for
	every $2\le t\le k-1$. Define an operation $\Phi_t$ on collections of subspaces by declaring
	$\Phi_t(  A)$ to be the set of all subspaces that occur as
	intersections of $r-t+1$ elements of $  A$ and have dimension
	$d(t-1)$.
	
	We claim that $\Phi_t( S_t)= S_{t-1}.$ First, let $J\subseteq [r]$ with $|J|=t-1$. Consider the $r-t+1$
	subspaces
	\[
		U_{J\cup\{i\}},\qquad i\in [r]\setminus J.
	\]
	Their intersection is $U_J$. Indeed, they all contain $U_J$, and if
	$i\neq j$ are not in $J$, then
	\[
		U_{J\cup\{i\}}\cap U_{J\cup\{j\}}=U_J.
	\]
	This follows since the subspaces indexed by $J\cup\{i,j\}$ are in direct sum, because
	$|J|+2=t+1\le k$. Hence $U_J\in \Phi_t( S_t)$. This proves $S_{t-1}\subseteq \Phi_t( S_t).$
	
	Conversely, suppose $W\in \Phi_t( S_t)$. Then $W=\bigcap_{I\in  T}U_I$ for some collection $ T$ of $r-t+1$ subsets $I\subseteq [r]$
	with $|I|=t$, and $\dim W=d(t-1).$ By Lemma~\ref{lem:combinatorics_general} there exists a subset $J\subseteq [r]$ with $|J|=t-1$ such that
	\[
		T=\{J\cup\{i\}:i\in [r]\setminus J\}.
	\]
	Therefore
	\[
		W=\bigcap_{i\in [r]\setminus J}U_{J\cup\{i\}}=U_J,
	\]
	so $W\in  S_{t-1}$. This completes the proof that $\Phi_t( S_t)= S_{t-1}.$
	
	The same argument applies to the primed collection, because
	$U'_1,\ldots,U'_r$ are also $d$-dimensional and have Kruskal rank at
	least $k$. Hence $\Phi_t( S'_t)= S'_{t-1}$ for every $2\le t\le k-1$.
	
	By assumption, $S_{k-1}= S'_{k-1}.$ Applying $\Phi_{k-1}$ to both sides gives $S_{k-2}= S'_{k-2}.$
	Applying $\Phi_{k-2}$, then $\Phi_{k-3}$, and so on, we eventually get $S_1= S'_1.$
	But
	\[
		S_1=\{U_1,\ldots,U_r\},\qquad
		S'_1=\{U'_1,\ldots,U'_r\},
	\]
	so the proof is complete.
\end{proof}

We now use Lemma~\ref{thm:combinatorics_general} along with contraction varieties to give a more geometric proof of Kruskal's theorem.

\begin{theorem}[Kruskal's theorem]\label{thm:kruskal}
	Suppose that $T=\sum_{i=1}^r a_i\otimes b_i \otimes c_i \in \mathbb{K}^{n_1} \otimes \mathbb{K}^{n_2} \otimes \mathbb{K}^{n_3}$, where the vectors $\{a_1,\dots, a_r\}$ have Kruskal rank $k_a$, the vectors $\{b_1,\dots, b_r\}$ have Kruskal rank $k_b$, and the vectors  $\{c_1,\dots, c_r\}$ have Kruskal rank $k_c$. If $2r \leq k_a+k_b+k_c-2$, then $\text{rank}(T)=r$ and this is the unique tensor rank decomposition of $T$.
\end{theorem}

\begin{proof}
	Let $m=r-k_c+1$, and note that $r+m \leq k_a+k_b-1$. By Proposition~\ref{prop:standard-kruskal}, the only linear combinations $\sum_{i=1}^r \alpha_i a_i b_i^T$ of rank $\leq m$ are those for which $\omega(\alpha)\leq m$. From here, it is clear that
	\begin{align}\label{eq:contraction_kruskal}
		X_{m}(T)=\bigcup_{i_1< \dots < i_{k_c-1}} \langle c_{i_1}, \dots, c_{i_{k_c-1}} \rangle^{\perp}.
	\end{align}
	So this is the decomposition of $X_m(T)$ into irreducible components. In particular, $\dim(X_m(T))=n_3-k_c+1$. 
	
	Let
	\[
	T=\sum_{i=1}^{r'} a'_i\otimes b'_i\otimes c'_i
	\]
	be a decomposition of $T$ into $r'\le r$ product vectors. We first prove that $r'=r$.  If $r'\le m$, then every contraction of $T$ has rank at
	most $m$, so $X_m(T)=\Kbar^{n_3}$, contradicting
	$\dim X_m(T)=n_3-r+m<n_3$, where the strict inequality follows from $k_c \geq 2$, which is an easy consequence of the assumption $2r \leq k_a+k_b+k_c-2$.  Thus $r'>m$. For every subset $J\subseteq [r']$ of cardinality $r'-m$, we have
	\[
	\langle c'_j:j\in J\rangle^\perp\subseteq X_m(T),
	\]
	because a form vanishing on these $r'-m$ vectors leaves at most $m$
	rank-one summands in the contraction.  Hence
	\[
	\dim X_m(T)\ge n_3-(r'-m)=n_3-r'+m.
	\]
	But $\dim X_m(T)=n_3-r+m$, so $r'\ge r$.  Since $r'\le r$, we conclude
	that $r'=r$.
	
	Arguing similarly to before, it is clear that
	\begin{align}\label{eq:Xm1}
	X_{m-1}(T)
	=
	\bigcup_{\substack{I\subseteq [r]\\ |I|=k_c}}
	\langle c_i:i\in I\rangle^\perp.
	\end{align}
	We now use this to prove that $k'_c\ge k_c$. Suppose toward contradiction that $k'_c<k_c$. Then there exists a
	subset $J\subseteq [r]$ with $|J|=k_c$ such that the set $\{c'_j:j\in J\}$ is linearly dependent. Hence,
	\[
	\dim \langle c'_j:j\in J\rangle\le k_c-1,
	\]
	so
	\[
	\dim \langle c'_j:j\in J\rangle^\perp\ge n_3-k_c+1.
	\]
	But every $f\in \langle c'_j:j\in J\rangle^\perp$ annihilates the $k_c$
	summands indexed by $J$ in the primed decomposition. Therefore $T_f$ is
	a sum of at most
	$
	r-k_c=m-1
	$
	rank-one matrices, so
	\[
	\langle c'_j:j\in J\rangle^\perp\subseteq X_{m-1}(T).
	\]
	This is impossible, because $X_{m-1}(T)\subseteq \Kbar^{n_3}$ is a finite union of linear
	subspaces of dimension $n_3-k_c$, whereas
	$\langle c'_j:j\in J\rangle^\perp$ is a linear subspace of dimension at
	least $n_3-k_c+1$, and $\Kbar$ is infinite. Therefore $k'_c\ge k_c$.
	
	It is clear that
	\begin{align}\label{eq:xtprime}
		X_m(T) \supseteq \bigcup_{i_1< \dots < i_{k_c-1}} \langle c'_{i_1}, \dots, c'_{i_{k_c-1}} \rangle^{\perp}.
	\end{align}
	Since $k_c' \geq k_c$, the spaces $\langle c'_{i_1}, \dots, c'_{i_{k_c-1}} \rangle$ are all distinct. Since the irreducible components appearing in the righthand side of~\eqref{eq:xtprime} have the same dimension and number as the irreducible components of $X_m(T)$, it follows that~\eqref{eq:xtprime} holds with equality.  By Lemma~\ref{thm:combinatorics_general} (applied using the field $\Kbar$) it holds that
	\begin{align*}
		\{\langle c_1 \rangle ,\dots, \langle c_r \rangle\} = \{\langle c_1' \rangle ,\dots, \langle c_r' \rangle\}.
	\end{align*}

	An analogous argument shows that the one-dimensional spaces $\langle a_i \rangle$ and $\langle b_i \rangle$ are uniquely determined. Hence, any rank $\leq r$ decomposition of $T$ must be of the form $T=\sum_{i=1}^r \alpha_i a_{\sigma(i)} \otimes b_{\tau(i)} \otimes c_{i}$ for some non-zero scalars $\alpha_i$ and permutations $\sigma, \tau \in \mathfrak{S}_r$.

	Fix any index $i'\in [r]$. We now verify that $\sigma=\tau$. If $m=1$ then we can let $f \in \Kbar^{n_3}$ be such that $\langle f,c_j\rangle =0$ for all $j \neq i'$ and $\langle f,c_{i'}\rangle \neq 0$. Then $\langle f, c_{i'} \rangle a_{i'} b_{i'}^T = \langle f, c_{i'}  \rangle \alpha_{i'} a_{\sigma(i')} b_{\tau(i')}^T$, so $\sigma=\tau=id$ since $k_a, k_b \geq 2$. Suppose now that $m > 1$, and suppose toward contradiction that $\sigma(i')\neq \tau(i')$ for some $i' \in [r]$. Let $S \subseteq[r]$ be a subset of size $m$ for which $i'\in S$ and $\sigma(i') \in S$. Let $f \in \Kbar^{n_3}$ be such that $\langle f,c_j\rangle =0$ for all $j \notin S$ and $\langle f,c_j\rangle \neq 0$ for all $j \in S$. Then
	\begin{align*}
		T_f = \sum_{i \in S} \langle f,c_i\rangle a_i b_i^T = \sum_{i \in S} \langle f, c_i\rangle \alpha_i a_{\sigma(i)} b_{\tau(i)}^T.
	\end{align*}
	Hence,
	\begin{align}\label{eq:sigma}
		a_{\sigma(i')} (\langle f,c_{\sigma(i')}\rangle b_{\sigma(i')}^T-\langle f,c_{i'} \rangle \alpha_{i'} b_{\tau(i')}^T)+&\sum_{i \in S \setminus \{\sigma(i')\}} \langle f, c_i\rangle a_i b_i^T\\
		& = \sum_{i \in S \setminus \{i'\}} \langle f, c_i\rangle  \alpha_i a_{\sigma(i)} b_{\tau(i)}^T.\nonumber
	\end{align}
	Note that the column vectors appearing in the lefthand side have Kruskal rank at least $\min\{m,k_a\}$, and the row vectors have Kruskal rank at least $\min\{m,k_b-1\}$. It is easily verified that
	\begin{align*}
		2m-1 \leq \min\{m,k_a\}+\min\{m,k_b-1\}-1.
	\end{align*}
	By Proposition~\ref{prop:standard-kruskal}, the only rank $\leq m-1$ elements of the span of the rank-one matrices appearing on the lefthand side of~\eqref{eq:sigma} are linear combinations of $\leq m-1$ of these matrices. This contradicts~\eqref{eq:sigma}. Hence, $\sigma=\tau$.

	By Proposition~\ref{prop:standard-kruskal}, the vectors $\{a_i \otimes b_i : i\in [r]\}$ are linearly independent. Since $\sum_{i=1}^r a_i \otimes b_i \otimes (c_i-\alpha_{\sigma^{-1}(i)} c_{\sigma^{-1}(i)}) = 0$, it follows that $c_i-\alpha_{\sigma^{-1}(i)} c_{\sigma^{-1}(i)}=0$. Since $k_c\geq 2$, this implies $\sigma=id$ and $\alpha_i=1$, completing the proof.
\end{proof}

\begin{remark}
	The works~\cite{Lovitz_Petrov_2023} and~\cite{gubkin2024unique} prove generalizations of Kruskal's theorem.  It is unclear to the authors how to extend the above geometric proof of Kruskal's theorem to give alternate proofs of either of these generalizations.
	For example, let us consider the generalization proven in~\cite{gubkin2024unique}:
	
	\begin{theorem}\label{thm:gubkin}
		Let $T=\sum_{i=1}^r a_i \otimes b_i \otimes c_i$. Then uniqueness holds if $k_c \geq 2$ and for any $\alpha \in \mathbb K^r$, it holds that
		\begin{equation}\label{eq:rank-condition-c}
			\text{{rank}}\Big(\sum_{i=1}^r \alpha_i\, a_i b_i^T\Big)\ \ge\ \min\{\omega(\alpha),\, \tilde{m}+1\},
		\end{equation}
		where $\tilde{m}=r-d_c+1$, $d_c:=\dim\langle c_1,\dots, c_r\rangle$, and $\omega(\alpha)$ denotes the number of nonzero entries in $\alpha$.
	\end{theorem}
	
	We briefly describe the difficulty in proving this theorem using our contraction variety approach. Under these assumptions, it is easy to see that
	\begin{align*}
		X_{\tilde{m}}(T)=\bigcup_{i_1<\dots<i_{d_c-1}} \langle c_{i_1},\dots, c_{i_{d_c-1}} \rangle^{\perp}.
	\end{align*}
	Recall that a similar observation~\eqref{eq:contraction_kruskal} appears in our proof of Kruskal's theorem, with $d_c$ replaced by $k_c$. In that case, we were able to prove that this union of subspaces uniquely determines the one-dimensional spaces $\langle c_i \rangle$. Unfortunately, this no longer holds when $k_c$ is replaced by $d_c$: Consider the vectors (over $\mathbb C$, for example)
	\begin{align*}
		c_1&=(0,0,1), \qquad c'_1=(1,0,-1)\\
		c_2&=c_2'=(0,1,-1)\\
		c_3&=c_3'=(0,1,0)\\
		c_4&=c_4'=(1,-1,-1)\\
		c_5&=c_5'=(1,-1,0).
	\end{align*}
	Then $d_c=d_{c'}=3$, $k_c=k_{c'}=2$, and
	\begin{align*}
		\{\langle c_i, c_j \rangle : 1\leq i<j\leq 5\}=\{\langle c'_i, c'_j \rangle : 1\leq i<j\leq 5\}.
	\end{align*}
	Indeed, this follows from
	\begin{align*}
		\langle c_1, c_2 \rangle &= \langle c_1, c_3 \rangle=\langle c_2, c_3 \rangle \\
		\langle c_1, c_4 \rangle &= \langle c_1, c_5 \rangle = \langle c_4, c_5 \rangle\\
		\langle c_1', c_2 \rangle &= \langle c_1', c_5 \rangle = \langle c_2, c_5 \rangle \\
		\langle c_1',c_3 \rangle &= \langle c_1', c_4\rangle = \langle c_3, c_4\rangle.
	\end{align*}
	However, $\{\langle c_1 \rangle, \dots, \langle c_5 \rangle \} \neq \{\langle c'_1 \rangle,\dots, \langle c'_5\rangle \}$.
\end{remark}

	\section{Border rank=rank for Kruskal tensors}\label{sec:border-rank=rank}
In this section, we prove the following theorem. 
\begin{theorem}\label{thm:border-rank=rank}
	Let $ T\in \K^{n_1\times n_2\times n_3} $ be a Kruskal tensor of rank $ r $. Then, the border rank of $ T $ is also $ r $. 
\end{theorem}

\noindent
We begin with a small lemma. 
\begin{lemma} \label{lem:contraction-variety-dimension}
	Let $ T  \in \K^{n_1\times n_2\times n_3}$ be a tensor with decomposition $T= \sum_{i = 1}^{r} a_i\otimes b_i\otimes c_i $ and let $ m\le k_a+k_b-(r+1) $ be bounded as in \Cref{prop:standard-kruskal}. Assume that  $r-m\le k_c$. Then, the dimension of $ X_{m}^{(3)}(T) $ equals $ n_3-\max\{r-m, 0\} $. Furthermore, $ X_{m}^{(3)}(T) $ is an arrangement of equidimensional subspaces. 
\end{lemma}
\begin{proof}
	Write $ r = \rk (T) $. If $ r < m $ then trivially $ X_{m}^{(3)}(T) = \Kbar^{n_3} $. Otherwise, by \Cref{prop:standard-kruskal}, the matrix space $ \mathcal{L} = \langle a_ib_i^{T} \mid i = 1,\ldots,r \rangle $ intersects the variety of matrices of rank $ \le m $ precisely in those matrices that are $ m $-sparse linear combinations of $ \{a_ib_i^{T} \mid i = 1,\ldots,r \} $. However, a contraction $ T_f = \sum_{i = 1}^{r} \langle f, c_i \rangle a_i b_{i}^{T} $ of $ T $ is an $ m $-sparse combination, if and only if $ \langle f, c_i \rangle = 0 $ for at least $ r-m $ values of $ i\in \{1,\ldots,r\} $. 
	From that, one readily observes that the contraction variety of $ T $ is a union of subspaces given by 
	\begin{align}\label{eq:kruskal-contraction-variety}
		X_m(T) = \bigcup_{I\subseteq [r], |I| = r-m} \langle c_i \mid i\in I \rangle^{\perp}.
	\end{align}
	The codimension of each irreducible component is thus at most $ r-m $, with equality as long as $ r-m\le k_c $, and thus the dimension of $ X_m(T) $ equals $ n_3-r+m $.
\end{proof}

Another small lemma is the flattening bound for Kruskal tensors. This is well known, but needed in the proof of \Cref{thm:Kruskal_border_rank_precise}. 
\begin{lemma}\label{lem:flattening-bound}
	Let $ T = \sum_{i = 1}^{r} a_i \otimes b_i\otimes c_i $ have k-ranks $ k_a, k_b, k_c $ and assume that $ 2r\le k_a+k_b+k_c-1 $. Then, the border rank $ s $ of $ T $ is at least $ \max\{k_a, k_b, k_c\} $.
\end{lemma}
\begin{proof}
	By flattening, we can view $ T = \sum_{i = 1}^{r} a_i \otimes b_i\otimes c_i $ as a matrix $ \sum_{i = 1}^{r}  a_i (b_i\otimes c_i)^{T} $, whose rank gives a lower bound on $ s $. It follows from \Cref{prop:standard-kruskal} that the matrices $ (b_ic_i^{T})_{i=1,\ldots,r} $ are linearly independent. Therefore, $ s\ge k_a $. By symmetry, $ s\ge  \max\{k_a, k_b, k_c\}. $
\end{proof}

\noindent
Let us now prove \Cref{thm:border-rank=rank}. More precisely, we are proving a just slightly stronger theorem, where Kruskal's condition is weakened by one. 
\begin{theorem}\label{thm:Kruskal_border_rank_precise}
	Let $ T\in \K^{n_1\times n_2\times n_3} $ be a tensor with a decomposition $ T = \sum_{i = 1}^{r} a_i\otimes b_i\otimes c_i $ such that $ 2r\le k_a + k_b + k_c - 1 $. Then, the border rank of $ T $ is $ r $. 
\end{theorem}
\begin{proof}
	The border rank of $ T $ over $ \Kbar $ is a lower bound to the border rank of $ T $ over $ \K $. Since the goal of this proof is to show a lower bound on the border rank of $ T $, it suffices thus to consider the case where $ \K $ is algebraically closed. We will assume $ \K = \Kbar $ in the following. 
	For $ m = r-k_c $, we consider the contraction variety $ X_m^{(3)}(T) $ of $ T $ with respect to the third tensor index. We wish to apply \Cref{lem:contraction-variety-dimension}. We know that $ 2r+1\le k_a+k_b+k_c $, and this inequality can easily be transformed to $ m = r-k_c \le k_a + k_b - (r+1) $. Furthermore, this choice of $m$ satisfies $r-m\le k_c$, which is the other condition of \Cref{lem:contraction-variety-dimension}. Hence, \Cref{lem:contraction-variety-dimension} implies that $ X_{m}^{(3)}(T) $ is an arrangement of equidimensional subspaces of dimension $ n_3-r+m $ as given by equation \eqref{eq:kruskal-contraction-variety}, so that 
	\begin{align*}
		X_m(T) = \bigcup_{I\subseteq [r], |I| = r-m} \langle c_i \mid i\in I \rangle^{\perp}.
	\end{align*}
	Let $ s $ denote the border rank of $ T $ and denote by $ \sigma_s \subseteq \K^{n_1\times n_2\times n_3}$ the set of tensors of border rank at most $ s $. By \Cref{lem:flattening-bound}, we must have $ s\ge \max\{k_a, k_b, k_c\} \ge m+1$. Here, the last bound follows from the weakened Kruskal inequality, since $ k_a\ge r-k_b+r-k_c+1 \ge r-k_c + 1= m+1 $. We are to show that for each tensor $ S \in \sigma_s $, the dimension of $ X_m(S) $ is at least $ n_3-s+m $. 
	This then shows that $ s = r $, as $ s < r $ would yield the contradiction $ n_3-s+m > n_3-r+m = \dim X_m(T) $. 
	
	For generic $ (a_i', b_i', c_i')_{i=1,\ldots,s} \in \K^{sn_1+sn_2+sn_3} $, we have $ k_{a'} = \min\{n_1, s\}\ge k_a$. Similarly, $k_{b'} \ge k_b $ and $k_{c'} \ge k_c $. Therefore, $ S = \sum_{i = 1}^{s} a_i'\otimes b_i'\otimes c_i' $ satisfies the conditions of \Cref{lem:contraction-variety-dimension}: Indeed, we see that $ m\le k_{a'} + k_{b'} - (s+1)$, since $ s\le r $ and $ m\le k_a + k_b - (r+1) $. Similarly, we see that $ s-m\le k_{c'} $, since $ s\le r $, $ k_{c'}\ge k_c $ and $ m = r-k_c $. Hence, \Cref{lem:contraction-variety-dimension} applied to $ S $ shows that $ X_m(S) $ is an arrangement of equidimensional subspaces of dimension $ n_3-s+m $. 
	
	Note that the contraction varieties $ \P(X_m(S)) $ are the fibers of the variety
	\begin{align*}
		\{([S], [f])\in \P(\sigma_s) \times \P(\K^{n_3}) \mid \rk S_{f} \le m \} 
	\end{align*}
	under projection to the $ S $-coordinates. By the semicontinuity theorem (see \cite[Exercise 3.22, Chapter II]{hartshorne2013algebraic}), the set of tensors $ S $ for which $ X_m(S) $ has dimension at least $ n_3-s+m $ is Zariski closed in $ \sigma_s $. In the previous paragraph, we showed that this set contains a nonempty Zariski open set. Therefore, it must be all of $ \sigma_s $. This concludes the proof. 
\end{proof}

\begin{remark}\label{rem:choice-of-m}
	In the particular case where $ \K = \C $ and the Kruskal rank $ k_c = n_{3} $ is full, we may circumvent the fiber dimension theorem in favour of a more elementary argument. Indeed, in this case, for $ m = r-k_c $, we see that $ X_m(T) $ has dimension $ n_{3}-r+m = n_{3}-k_c = 0 $. A general tensor $ S $ of rank $ m\le s \le r $ will satisfy Kruskal's inequality, so that $ X_m(S) $ has dimension $ n_{3}-s+m = r-s $, which for $ s<r $ is strictly positive. 
	If $ s $ is the border rank of $ T $, then we find a curve $ S_{\varepsilon} $ of Kruskal tensors of rank $ s $ with $ S_{\varepsilon} \to T $ as $ \varepsilon\to 0 $. If the border rank $ s $ is strictly smaller than the rank $ r $, then $ X_m(S_{\varepsilon}) $ is nonempty. Correspondingly, we find a sequence of contractions $ f_{\varepsilon} \in X_{m}(S_{\varepsilon}) $ as $ \varepsilon\to 0 $. Without loss of generality, we normalise the contractions $ f_{\varepsilon} \in \mathbb C^{n_{3}}$ to unit length with respect to some standard Euclidean norm. By compactness of the unit sphere, we find a convergent subsequence $ (f_{\varepsilon_{i}})_{i\in \N} $, which has a limit point $ f $ lying on the unit sphere of $ \mathbb C^{n_{3}} $. Clearly, $ \rk (T_{f}) \le \lim_{i\to \infty}\rk ((S_{\varepsilon_i})_{f_{\varepsilon_i}})$, since a sequence of matrices can only decrease rank in the limit. Therefore, $ f\in X_m(T) = \{0\} $, which is a contradiction. 
\end{remark}

	\section{A Kruskal's theorem for skew decompositions}\label{sec:alternating-kruskal}
{In this section, we prove analogues to the above results for the case of skew decompositions of alternating tensors. In particular, we prove an alternating Kruskal's theorem (\Cref{thm:alternating_kruskal}), and that border skew rank = skew rank for alternating Kruskal tensors (\Cref{thm:border-skew-rank=skew-rank}). 
For technical convenience, we restrict ourselves to fields whose characteristic is neither two nor three. Throughout the section, all fields are therefore understood to be in characteristic zero or at least five.  The reason is that over fields of characteristic two, one has to carefully distinguish between alternating 2-forms and skew-symmetric 2-forms, whereas these notions are equivalent in higher characteristics. A similar problem occurs with $ 3 $-forms in characteristic $ 3 $. By excluding characteristic $ 3 $, we are also able to write $ a_1\wedge a_2\wedge a_3 = \frac{1}{3!} \sum_{\sigma\in \mathfrak{S}_3} \operatorname{sgn}(\sigma) a_{\sigma(1)}\otimes a_{\sigma(2)}\otimes a_{\sigma(3)} $, where $ \mathfrak{S}_3 $ is the group of permutations of $ [3] $ and $ \operatorname{sgn} $ denotes the sign of a permutation.}

\begin{lemma}\label{lem:bipartite_alt}
	Let $a_1 \wedge b_1, \dots, a_r \wedge b_r \in \Lambda^2(\mathbb K^n)\setminus \{0\}$ be such that the collection of subspaces $\{U_i=\langle a_i , b_i \rangle : i=1,\dots, r\}$ has Kruskal rank $k$. Suppose further that $r+m \leq 2k-1$ for some non-negative integer $m$.  If
	\[
	\mathrm{rank}(\sum_{i=1}^r \alpha_i \; a_i \wedge b_i) \leq 2m
	\] for some $\alpha=(\alpha_1,\dots, \alpha_r)\in \Kbar^r$, then $\omega(\alpha) \leq m$, where $\omega(\alpha)$ denotes the number of nonzero entries in $\alpha$.
\end{lemma}
\begin{proof}
	Suppose toward contradiction that there exists $\alpha \in \Kbar^r$ such that $\sum_{i=1}^r \alpha_i a_i \wedge b_i$ has rank $\leq 2m$ but $\omega(\alpha)> m$. If $\omega(\alpha) \leq k$, then this linear combination has rank $2 \omega(\alpha)>2m$, a contradiction.
	
	Suppose now that $\omega(\alpha) > k$, and note that $\omega(\alpha) - k \leq r-k \leq k -m -1$. Let $I \subseteq [r]$ be a subset of size $k$ for which $\alpha_i \neq 0$ for all $i \in I$. Then
	\begin{align*}
		\mathrm{rank}(\sum_{i=1}^r \alpha_i \; a_i \wedge b_i) &\geq \mathrm{rank}(\sum_{i \in I} \alpha_i \; a_i \wedge b_i)-\mathrm{rank}(\sum_{i \notin I} \alpha_i \; a_i \wedge b_i)\\
		&\geq 2k -2(k-m-1) = 2(m+1),
	\end{align*}
	a contradiction. This completes the proof.
\end{proof}

\begin{theorem}\label{thm:alternating_kruskal}
	Let $T = \sum_{i=1}^r a_{i1} \wedge a_{i2} \wedge a_{i3} \in \Lambda^3(\mathbb K^n)$ be an alternating tensor, and let $U_i=\langle a_{i1},a_{i2},a_{i3}\rangle$. Suppose the collection of subspaces $U_1,\dots, U_r$ has Kruskal rank $k$, and $2r \leq 3k-2$. Then $T$ has skew rank $r$ and this is the unique skew rank decomposition of $T$.
\end{theorem}
\begin{proof}
	Let $m\in \{r-k+1,\dots, 2k-r-1\}$ be any integer, and let $h=r-m$. We claim that
	\begin{align}\label{eq:alternating_XT}
		X_{2m}(T)= \bigcup_{i_1<\dots<i_{h}} \langle U_{i_1},\dots, U_{i_{h}} \rangle^{\perp}.
	\end{align}
	For the containment $\supseteq$, note that if $f \in\langle U_{i_1},\dots, U_{i_{h}} \rangle^{\perp}$, then $T_f$ is a linear combination of at most $r-h = m$ terms of skew rank 1, hence $f \in X_{2m}(T)$. Conversely, suppose that $T_f$ has rank $\leq 2 m$. Note that $T_f= \sum_{i} M_{if}$ where the $M_{if}$ are wedge products with image $U_i \cap \langle f \rangle^{\perp}$. Let $I\subseteq [r]$ be the set of indices for which $f \notin U_i^{\perp}$. Since the collection of subspaces $U_1,\dots, U_r$ has Kruskal rank $k$, it follows that the collection of subspaces $U_i \cap \langle f \rangle^{\perp}$ as $i$ ranges over $I$ has Kruskal rank at least $\min\{|I|,k\}$. If $|I| \leq m$ then $f$ is orthogonal to at least $r-m=h$ subspaces $U_i$, so $f$ is contained in the righthand side of~\eqref{eq:alternating_XT}. Now suppose $|I|>m$. If $|I| \geq k$, then this contradicts Lemma~\ref{lem:bipartite_alt} since $r+m \leq 2k-1$. If $|I|<k$, then $T_f$ has rank $2|I|>2m$, a contradiction. This proves the containment $\subseteq$.
	
	It is clear that $h\leq k-1$, hence the spaces $\langle U_{i_1},\dots, U_{i_{h}} \rangle^{\perp}$ are pairwise non-equal, hence~\eqref{eq:alternating_XT} is the decomposition of $X_{2m}(T)$ into irreducible components.
	
	Let \[T=\sum_{i=1}^{r'} a_{i1}' \wedge a_{i2}' \wedge a_{i3}'\] be an $r'\leq r$-term decomposition of $T$. We first prove that $r'=r$. Let $U_i':=\langle a_{i1}',a_{i2}',a_{i3}'\rangle.$ Recall from
\eqref{eq:alternating_XT} that
\[
    \dim X_{2m}(T)=n-3h=n-3(r-m).
\]

If $r'\leq m$, then for every $f\in \Kbar^n$, the contraction $T_f$ is
a sum of at most $r'\leq m$ skew rank-one matrices.  Hence
$X_{2m}(T)=\Kbar^n$, contradicting
\[
    \dim X_{2m}(T)=n-3h<n.
\]
Therefore $r'>m$. Now let $J\subseteq [r']$ be any subset of cardinality $r'-m$.  If $f\in \langle U_j' : j\in J\rangle^\perp,$
then the summands indexed by $J$ vanish after contraction by $f$, so
$T_f$ is a sum of at most $m$ skew rank one (hence rank two) matrices.  Thus $\langle U_j' : j\in J\rangle^\perp\subseteq X_{2m}(T).$
Since
\[
    \dim \langle U_j' : j\in J\rangle\leq 3(r'-m),
\]
we get
\[
    \dim X_{2m}(T)\geq n-3(r'-m).
\]
Combining this with $\dim X_{2m}(T)=n-3(r-m)$, we obtain $r'\geq r$.
Since $r'\leq r$, it follows that $r'=r$.

Arguing similarly to before, it is clear that
\begin{align}\label{eq:Xm1alt}
X_{2(r-k)}(T)
=
\bigcup_{\substack{I\subseteq [r]\\ |I|=k}}
\langle U_{i_1},\dots, U_{i_{k}} \rangle^{\perp}.
\end{align}
Note that $r=k$ is possible. Let $k'$ be the Kruskal rank of $\{U'_1,\dots, U'_r\}$. We now use~\eqref{eq:Xm1alt} to prove that $k'\ge k$. Suppose toward contradiction that $k'<k$. Then there exists a
subset $J\subseteq [r]$ with $|J|=k$ such that the set $\{U'_j:j\in J\}$ is not in direct sum. Hence
\[
\dim \langle U'_j:j\in J\rangle^\perp\ge n-3k+1.
\]
But every $f\in \langle U'_j:j\in J\rangle^\perp$ annihilates the $k$
summands indexed by $J$ in the primed decomposition. Therefore $T_f$ is
a sum of at most $r-k$ skew rank one matrices, so
\[
\langle U'_j:j\in J\rangle^\perp\subseteq X_{2(r-k)}(T).
\]
This is impossible, because $X_{2(r-k)}(T)$ is a finite union of linear
subspaces of dimension $n-3k$, whereas
$\langle U'_j:j\in J\rangle^\perp$ is a linear subspace of dimension at
least $n-3k+1$. Therefore $k'\ge k$.

Since $h \leq k-1\leq k'-1$, the spaces $\langle U'_{i_1}, \dots, U'_{i_{h}} \rangle$ are all distinct. Clearly,
	\begin{align}\label{eq:xtprimealt}
		X_{2m}(T) \supseteq \bigcup_{i_1<\dots<i_{h}} \langle U_{i_1}',\dots, U_{i_{h}}' \rangle^{\perp}.
	\end{align}
	Since the irreducible components appearing in the righthand side of~\eqref{eq:xtprimealt} have the same dimension and number as the irreducible components of $X_{2m}(T)$, it follows that~\eqref{eq:xtprimealt} holds with equality.  By Lemma~\ref{thm:combinatorics_general} (applied to the $h$-spans), it holds that
	\begin{align*}
		\{U_1 ,\dots, U_r\} = \{U'_1 ,\dots, U'_r\}.
	\end{align*}
	By permuting terms in the primed decomposition of $T$, we may assume $U_i'=U_i$, and hence $a_{i1}' \wedge a_{i2}' \wedge a_{i3}' = \alpha_i \; a_{i1} \wedge a_{i2} \wedge a_{i3}$ for some $\alpha_i \neq 0$. To complete the proof, we just need to show that $\alpha_i=1$ for all $i$. For this, it suffices to prove that $\{a_{i1} \wedge a_{i2} \wedge a_{i3} : i=1,\dots, r\}$ is linearly independent. Let $f \in \Kbar^n$ be any vector that is not orthogonal to any $U_i$, and note that $(a_{i1} \wedge a_{i2} \wedge a_{i3})_f$ is a matrix of rank two with image contained in $U_i$. It suffices to prove that these $r$ matrices are linearly independent. Note that the images of these matrices still have Kruskal rank at least $k$. Since $2r \leq 3k-2$, we have $r \leq 2k-1 + (k-r-1) \leq 2k-1$, so by Lemma~\ref{lem:bipartite_alt} these matrices are linearly independent. This completes the proof.
\end{proof}

\subsection{The inequality appearing in our alternating Kruskal's theorem is sharp}
It was shown in \cite{derksen2013kruskal} that the inequality appearing in Kruskal's theorem is sharp. We use Derksen's example to prove sharpness also for the inequality appearing in our alternating Kruskal's theorem.   

\begin{theorem}\label{thm:alternating_sharp}
	Let $\mathbb K$ be a field with more than $2r$ elements, or a finite field of characteristic at least $2r$. If $2r = 3k-1$, then there exists a tensor $T=\sum_{i=1}^r a_i \wedge b_i \wedge c_i$ for which the Kruskal rank of the subspaces $U_i=\langle a_i, b_i, c_i\rangle$ is at least $k$, and this is not the unique skew-rank decomposition of $T$. 
\end{theorem}
\begin{proof}
	By~\cite[Theorem 2]{derksen2013kruskal} there exists $q \leq 2r$ and product tensors $\{a_{i} \otimes b_{i} \otimes c_{i}: i=1,\dots, q\}$ for which $\sum_{i=1}^q a_{i} \otimes b_{i} \otimes c_{i}=0$ and $k_a,k_b,k_c \geq k$. Note that $2r+1 = 3k \leq 2q+1$, where the equality is by assumption and the inequality follows from Kruskal's theorem. Hence $q \geq r$ and we can write
	\begin{align}\label{eq:non-identifiable-kruskal}
	\sum_{i=1}^r a_i \otimes b_i \otimes c_i = \sum_{i=r+1}^q (-a_i) \otimes b_i \otimes c_i.
	\end{align}
	Let
	\begin{align*}
	\tilde{a}_i := a_i \otimes (1,0,0),\quad \tilde{b}_i := b_i \otimes (0,1,0),\quad \tilde{c}_i := c_i \otimes (0,0,1).
	\end{align*}
	We see that taking the Kronecker product of \eqref{eq:non-identifiable-kruskal} with $e_1\otimes e_2 \otimes e_3$ on both sides yields the identity
	\begin{align}\label{eq:non-identifiable-kruskal-2}
	    \sum_{i=1}^r \tilde{a}_i \otimes \tilde{b}_i \otimes \tilde{c}_i = \sum_{i=r+1}^q (-\tilde{a}_i) \otimes \tilde{b}_i \otimes \tilde{c}_i
	\end{align}
	Antisymmetrizing \eqref{eq:non-identifiable-kruskal-2} yields
	\begin{align}\label{eq:alternating_sharp}
	\sum_{i=1}^r \tilde{a}_i \wedge \tilde{b}_i \wedge \tilde{c}_i = \sum_{i=r+1}^q (-\tilde{a}_i) \wedge \tilde{b}_i \wedge \tilde{c}_i
	\end{align}
	It is not difficult to verify that the collection of subspaces $U_i:=\langle \tilde{a}_i, \tilde{b}_i, \tilde{c}_i \rangle$ has Kruskal rank at least $k$. Since $q-r \leq r$, this then shows that the lefthand side of~\eqref{eq:alternating_sharp} is not a unique skew rank decomposition, completing the proof.
\end{proof}

\subsection{Border skew rank = skew rank for alternating Kruskal tensors}
In a similar fashion to \Cref{sec:border-rank=rank}, we are also able to prove that alternating Kruskal tensors have their border skew rank equal to their skew rank. Recall that the {border skew rank} of a tensor $ T\in \Lambda^3(\K^n) $ is the least value $ s $ such that $ T $ lies in the $ s $-secant of the irreducible affine cone $ Y = \{a\wedge b\wedge c \mid a, b, c\in \K^n\} $. Note that $ \P(Y) $ is isomorphic to the Grassmannian variety $ \grass(3,n) $ of $ 3 $-dimensional subspaces of $ \K^n $. We proceed analogously to \Cref{sec:border-rank=rank} and we begin with a small lemma. 

\begin{lemma}\label{lem:contraction-variety-kruskal-skew}
	Let $ T\in \Lambda^3(\mathbb K^n) $ be an alternating tensor with decomposition $ T = \sum_{i = 1}^{r} a_i\wedge b_i\wedge c_i $. Let $k$ be the Kruskal rank of the subspaces $ U_i = \langle a_i, b_i, c_i \rangle $, and let $ m\le 2k-r-1 $ be bounded as in \Cref{lem:bipartite_alt}. In addition, assume that $ r-k\le m $. Then, the dimension of $ X_{2m}(T) $ equals $ n-3\max\{r-m, 0\} $. Furthermore, $ X_{2m}(T) $ is an arrangement of equidimensional subspaces. 
\end{lemma}
\begin{proof}
	Note that for each $ f\in \K^n $, $ T_f $ is a matrix which may be written as $ T_f = \sum_{i = 1}^{r} M_{if} $, where $M_{if}$ is a matrix of skew rank at most 1 with $ \im M_{if} = U_i \cap \langle f \rangle^{\perp} $ (note that $ M_{if}=0 $ if and only if $f\in U_i^{\perp}$).
	
	If $ r < m $, then we trivially have $ X_{2m}(T) = \overline{\mathbb K}^n $.  Otherwise, we claim that 
	\begin{align}\label{eq:skew-kruskal-contraction-variety}
		X_{2m}(T) = \bigcup_{I\subseteq [r], |I| = r-m} \langle U_i \mid i\in I \rangle^{\perp}.
	\end{align}
	For the containment ``$ \supseteq $'', note that clearly any vector $ f\in \K^n $ which is orthogonal to at least $ r-m $ of the spaces $ U_i $ leaves after contraction a matrix $ T_f $, which is a sum of at most $ m $ matrices $ M_{if} $ of skew rank 1. Therefore, $ T_f $ has skew rank at most $ m $, hence rank at most $2m$.
	
	For the containment ``$ \subseteq $'', assume that the rank of $ T_f $ is at most $ 2m $, or, equivalently, the skew rank of $ T_f $ is at most $ m $. Since $ U_1,\ldots,U_r $ have Kruskal rank $ k$ and $r+m \leq 2k-1$, \Cref{lem:bipartite_alt} shows that $ f $ is orthogonal to at least $ r-m $ spaces $ U_i $. 

	From \eqref{eq:skew-kruskal-contraction-variety}, it follows that the codimension of each irreducible component of $ X_{2m}(T) $ is at most $ 3(r-m) $. Since we assumed $ r-m\le k$, orthogonality to $ r-m $ of the spaces $ U_i $ imposes independent constraints and thus the dimension of $ X_{2m}(T) $ equals $ n-3(r-m)$. 
\end{proof}

Before we prove \Cref{thm:border-skew-rank=skew-rank}, we derive a first auxiliary bound for the border skew rank of alternating Kruskal tensors. 
\begin{lemma}\label{lem:flattening-bound-skew}
	Let $ T = \sum_{i = 1}^{r} a_i \wedge b_i\wedge c_i $ be such that the subspaces $ U_i = \langle a_i, b_i, c_i \rangle $ have k-rank $ k $ and $ 2r\le 3k-1 $. Then, the border skew rank $ s $ of $ T $ is at least $ k$. Furthermore, for $ m = r-k$, we also have $ s\ge k\ge 2m+1 $. 
\end{lemma}
\begin{proof}
	Just as in the proof of \Cref{lem:contraction-variety-kruskal-skew}, we write $ T_f = \sum_{i = 1}^{r} M_{if} $. There exists $ f $ such that the matrix $ T_f $ has rank $ 2k $ (and skew rank $ k $). This is seen by choosing $ f $  orthogonal to $ r-k $ spaces $ U_{i_1},\ldots,U_{i_{r-k}} $ and such that $ U_i\cap\langle f \rangle^{\perp} $ is two-dimensional for the other $ k $ spaces. This is possible because $r-k \leq 2k-r-1\leq k-1$. If $ T $ had border skew rank $ s < k \le r $, then $ T $ would lie in the Zariski closure $ \sigma_s $ of the set of tensors of skew rank $ s $. For each such tensor $ S $ it holds that $ X_{2k-2}(S) = \Kbar^n $.
	
	Note that the contraction varieties $ \P(X_{2k-2}(S)) $ are the fibers of the variety
	\begin{align*}
		\{([S], [f])\in \P(\sigma_s) \times \P(\Kbar^{n}) \mid \rk(S_{f}) \le 2k-2 \}
	\end{align*}
	under projection to the $ S $-coordinates. By the semicontinuity theorem (see \cite[Exercise 3.22, Chapter II]{hartshorne2013algebraic}), the set of tensors $ S $ for which $ X_{2k-2}(S) $ has dimension at least $ n $ is Zariski closed in $ \sigma_s $. Hence, $ X_{2k-2}(T) = \Kbar^n $, in contradiction with the contraction we described above for which $ T_f $ has rank $ 2k $. Therefore, $ s \ge k $. 
	Due to $ 2r\le 3k-1 $ we see $ 2m = 2r-2k\le k-1 $, so that $ k\ge 2m+1 $.
\end{proof}

\begin{theorem}\label{thm:border-skew-rank=skew-rank}
	Let $ T\in \Lambda^{3}(\K^n) $ be an alternating tensor with a decomposition $ T = \sum_{i = 1}^{r} a_i\wedge b_i\wedge c_i $. Let $ k $ denote the Kruskal rank of the collection of subspaces $ U_i = \langle a_i, b_i, c_i \rangle $. If $ 2r\le 3k - 1 $, then the border skew rank of $ T $ is $ r $. 
\end{theorem}
\begin{proof}
	The border skew rank of $ T $ over $ \Kbar $ is a lower bound to the border skew rank of $ T $ over $ \K $. It suffices thus to consider the case where $ \K $ is algebraically closed, which we will assume in the following. 
	For $ m := r-k $, we consider the contraction variety $ X_{2m}(T) $ of $ T $. We wish to apply \Cref{lem:contraction-variety-kruskal-skew}. We know that $ 2r\le 3k-1 $, and this inequality can easily be transformed to $ m = r-k \le 2k - (r+1) $. Furthermore, this choice of $m$ satisfies $r-m\le k$, which is the other condition of \Cref{lem:contraction-variety-kruskal-skew}. Hence, \Cref{lem:contraction-variety-kruskal-skew} implies that $ X_{2m}(T) $ is an arrangement of equidimensional subspaces of dimension $ n-3r+3m $ as given by equation \eqref{eq:skew-kruskal-contraction-variety}, so that 
	\begin{align*}
		X_{2m}(T) = \bigcup_{I\subseteq [r], |I| = r-m} \langle U_i \mid i\in I \rangle^{\perp}.
	\end{align*}
	Let $ s $ denote the border skew rank of $ T $ and denote by $ \sigma_s \subseteq \Lambda^3(\K^n)$ the set of tensors of border skew rank at most $ s $. We are to show that for each tensor $ S \in \sigma_s $, the dimension of $ X_{2m}(S) $ is at least $ n-3s+3m $. This then shows that $ s = r $, as $ s < r $ would yield the contradiction $ n-3s+3m > n-3r+3m = \dim X_{2m}(T) $. 
	By \Cref{lem:flattening-bound-skew}, we know that $ s \ge k \ge 2m+1\ge m$. 
	
	For generic $ (a_i', b_i', c_i')_{i=1,\ldots,s} \in \K^{3ns} $, we have that the collection $ U_i' = \langle a_i', b_i', c_i' \rangle $ of 3-dimensional subspaces has full k-rank $ k' = \min\{\lfloor{n/3}\rfloor, s\}\ge k$, and thus the tensor $ S = \sum_{i = 1}^{s} a_i'\wedge b_i'\wedge c_i' $ satisfies the conditions of \Cref{lem:contraction-variety-kruskal-skew}: Indeed, we see that $ m\le 2k' - (s+1)$, since $ k'\ge k$, $ s\le r $ and $ m\le 2k - (r+1) $. Similarly, we see that $ s-m\le k' $, since $ s\le r $, $ k'\ge k $ and $ m = r-k $. Hence, \Cref{lem:contraction-variety-kruskal-skew} applied to $ S $ shows that $ X_{2m}(S) $ is an arrangement of equidimensional subspaces of dimension at least $ n-3s+3m $, with equality if $ s\ge m $. 
		
	Note that the contraction varieties $ \P(X_{2m}(S)) $ are the fibers of the variety
	\begin{align*}
		\{([S], [f])\in \P(\sigma_s) \times \P(\K^{n}) \mid \rk(S_{f}) \le 2m \}
	\end{align*}
	under projection to the $ S $-coordinates. By the semicontinuity theorem (see \cite[Exercise 3.22, Chapter II]{hartshorne2013algebraic}), the set of tensors $ S $ for which $ X_{2m}(S) $ has dimension at least $ n-3s+3m $ is Zariski closed in $ \sigma_s $. In the previous paragraph, we showed that this set contains a nonempty Zariski open set. Therefore, it must be all of $ \sigma_s $. This concludes the proof. 
\end{proof}

\section{An algorithmic alternating Kruskal's theorem}\label{sec:algorithmic_alternating}
In this section, we present an algorithm to compute the skew rank decomposition of an alternating Kruskal tensor. Conceptually, it can be seen as a version of the algorithm of Domanov and de Lathauwer \cite{domanov2014canonical} or the recent \cite{BSS26} 
for alternating Kruskal tensors. It can also be seen as an alternative to Vannieuwenhoven's algorithm for skew rank \cite{Vannieuwenhoven_2026_Chiseling} which holds in the overcomplete setting of higher ranks, albeit at the price of higher computational complexity.

Conceptually, the algorithm follows the identifiability proof from \Cref{thm:alternating_kruskal}: For suitable ranks, contraction varieties of alternating Kruskal tensors are subspace arrangements. As in the proof of \Cref{thm:alternating_kruskal}, it suffices to ``learn'' the arrangement, by finding bases for each subspace. For literature on learning subspace arrangements, we refer to \cite{Ma_Yang_Derksen_Fossum_2008}. In our setting, we learn the arrangement from linear slices, i.e., from witness sets of the contraction variety. See \cite{Sommese_Verschelde_Wampler_2001} for background on witness sets. Throughout, we assume again that $ \operatorname{char} \K \notin \{2,3\}$. We start with a lemma on learning subspace arrangements from linear slices. 

\begin{lemma}\label{lem:parallel-slice-component-recovery}
	Let
	\[
		X=\bigcup_{I\in A}L_I\subseteq \Kbar^n
	\]
	be a finite union of distinct $t$-dimensional linear subspaces. Let $E$ be a subspace of codimension $t$ and let $F$ be a complementary
	subspace, so that $\Kbar^n=E\oplus F$. Assume that $ E\cap L_I = \{0\}$ for each $ I \in A$. 
	For $f\in F$, write $W_f:=X\cap (f+E)$. Then, choosing a generic basis $f_1,\ldots,f_t$ of $ F $, bases for the subspaces $L_I$ can be recovered in time $\mathcal{O}(nt|A|^3)$ from the finite ``witness'' sets
	\[
		W_{f_1},\ldots,W_{f_t}, W_{f_1+f_2},\ldots,W_{f_1+f_t}.
	\]
\end{lemma}

\begin{proof}
	Since $ E\cap L_I = \{0\} $, for every $f\in F$ the affine space $f+E$ meets
	$L_I$ in a unique point $ x_{I, f} $. Write $ \pi_F\colon E\oplus F\twoheadrightarrow F $ for the linear epimorphism with kernel $ E $. Then $ \pi_{F}|_{L_I}\colon L_I \to F $ is invertible due to $ E\cap L_I = \{0\} $ and it holds that $x_{I, f} = (\pi_F|_{L_I})^{-1}(f)$. In particular, $ x_{I, f} $ depends linearly on $ f $. We write $ B_{I}\colon F\to L_I $ as a shorthand for $ (\pi_F|_{L_I})^{-1} $. 
	
	As we are given $ W_{f_1},\ldots,W_{f_t} $ and $ W_{f_1+f_2},\ldots,W_{f_1+f_t} $, the goal is to match the points $ x_{I, f} $ from all the spaces $ W_{f_i} $ into ``clusters'', so that points $ x_{I, f_i} $ and $ x_{J, f_j} $ are matched together if and only if $ I = J $. We put two points $ x_{I, f_1} \in W_{f_1} $ and $ x_{J, f_i} \in W_{f_i} $ into the same cluster if and only if there exists a point $ z\in W_{f_1+f_i} $ such that $ x_{I, f_1} + x_{J, f_i} = z $. 
	
	Clearly, if $ I = J $ then such a point $ z $ exists, since one may take $ z = x_{I, f_1 + f_i} $. Conversely, we are to show that whenever such a point $ z \in  W_{f_1+f_i} $ exists, which must have the form $ z = x_{K, f_1 + f_i}$ for some $ K\in A $, then we must have $ I = J = K $. 
	However, this is clear by genericity of $ f_1,\ldots,f_t $. Indeed, if we had 
	\begin{align}\label{eq:subspace_recovery_condition}
		x_{I, f_1} + x_{J, f_i} = x_{K, f_1 + f_i}
	\end{align}
	for all $ f_1,\ldots,f_t $, then $ L_I + L_J \subseteq L_K $, since $ B_I, B_J $ are surjections onto $ L_I $ and $ L_J $, respectively. This readily implies $ L_I = L_J = L_K $, as all subspaces have the same dimension. Hence, with this matching rule we can compute the set of sets $ \{\{x_{I,f_1},\ldots,x_{I, f_t}\} \mid I\in A\}$. Now, note that each of these sets spans a space $ L_I $. Indeed, $ x_{I,f_i} =  B_I(f_i) $ and $ B_I $ is a vector space isomorphism. Since $ f_1,\ldots,f_t $ is a basis of $ F $, $ x_{I,f_1},\ldots,x_{I, f_t} $ must therefore be a basis of $ L_I $. This concludes the recovery procedure.
	
	Regarding the runtime, the matching step checks at most $|A|^2$ pairs $(x,y)\in W_{f_1}\times W_{f_i}$ for each $i\in \{2,\dots,t\}$. For each pair, it determines if $x+y\in W_{f_1+f_i}$. Adding two vectors and checking equality of two vectors
	takes time $\mathcal{O}(n)$. Equality needs to be checked for all $ |A| $ elements of $ W_{f_1+f_i} $. Hence determining if $x+y\in W_{f_1+f_i}$ takes time $\mathcal{O}(n|A|)$, yielding a total runtime of $\mathcal{O}(nt|A|^3)$.
	
\end{proof}

There is a small catch when applying the previous lemma. Of course, we do not know whether a given choice of $ E, F $ and $ f_1,\ldots,f_t $ will satisfy the conditions of \Cref{lem:parallel-slice-component-recovery}. Over $ \Kbar = \C $, one may construct a generic space $ E $ and generic $ f_1,\ldots,f_t $ by sampling bases from a normal distribution. Over arbitrary algebraically closed fields, we use the Schwartz-Zippel lemma, stated below. It guarantees that there will be many ``sufficiently generic'' vectors in a sufficiently large box. 

\begin{lemma}[Schwartz-Zippel \cite{Schwartz_1980}]\label{lem:Schwartz-Zippel}
	Let $ p\in \K[x_1,\ldots,x_n] $ be a polynomial of degree $ d $ and let $ S \subseteq \K $ be a finite subset. If $ a\in S^n $ is chosen uniformly, then the probability that $ p(a) = 0 $ is at most $ \frac{d}{|S|} $.
\end{lemma}

To apply the Schwartz-Zippel lemma, we need a bound on the degree of the explicit nondegeneracy conditions that were used in the proof of \Cref{lem:parallel-slice-component-recovery}. This is done in the following lemma. 

\begin{lemma}\label{lem:genericity-implied-by-schwartz-zippel}
	Let $ S\subseteq \Kbar $ be a finite set. If $ E = \langle e_1,\ldots,e_{n-t} \rangle $ and $ F = \langle f_1,\ldots,f_t \rangle $ are sampled by choosing $ (e_1,\ldots,e_{n-t},f_1,\ldots,f_t)\in S^{n\times n} $ uniformly at random, then the conditions of \Cref{lem:parallel-slice-component-recovery}, namely
	\begin{enumerate}[(a)]
		\item $ E\oplus F = \Kbar^n $,
		\item $ E\cap L_I = \{0\} $ for all $ I\in A $,
		\item $ f_1,\ldots,f_t $ form a basis of $ F $, 
		\item $ f_1,\ldots,f_t $ satisfy the matching rule. 
	\end{enumerate}
	are fulfilled with probability at least $ 1-\frac{4n^2|A|^3}{|S|}$. 
\end{lemma}
\begin{proof}
	We can construct a polynomial for each condition independently and then multiply the polynomials together. 
	Conditions (a) and (c) are both satisfied if $ \det(e_1,\ldots,e_{n-t},f_1,\ldots,f_t) \ne 0$. This is an (negated) equation of degree $ n $ in the $ n^2 $ entries of $ e_1,\ldots,e_{n-t},f_1,\ldots,f_t $. 
	Condition (b) is saying that the spaces $ E $ and $ L_I $ are in direct sum for each $ I $. Fixing any basis $ \ell_{I,1},\ldots,\ell_{I,t} $ of $ L_I $, it is implied by $ \det(e_1,\ldots,e_{n-t},\ell_{I,1},\ldots,\ell_{I,t}) \ne 0 $, which is an negated equation of degree $ n-t $ in the entries of $ e_1,\ldots,e_{n-t} $.  
	In total, Condition (b) contributes $ |A| $ equations of degree $ (n-t) $ in $ e_1,\ldots,e_{n-t} $, which can be multiplied to a single equation $ p(e_1,\ldots,e_{n-t})\ne 0 $ of degree $ |A|(n-t) $. 
	
	Multiplying the polynomials from steps (a)-(c) together, we get a polynomial $ p $ of degree $ |A|(n-t) + n $. 
		
	For (d), note that the matching rule is implied by the condition that $ B_{I}(f_1) + B_{J}(f_i) \notin L_K$ for all $ i\in \{2,\ldots,t\} $, where $ B_{I}\colon F\to L_I $ is again the shorthand for $ (\pi_F|_{L_I})^{-1} $. In fact, take any linear form $ \ell_K $ which vanishes on $ L_K $ but not on any of the other spaces of the arrangement. Then, the negated linear equation $ \ell_K(B_{I}(f_1) + B_{J}(f_i)) \ne 0 $ implies the matching rule. Over all parameters $ I, J, K \in A $ (excluding the diagonal $ I = J = K $) and $ i\in \{2,\ldots,t\} $, these are $ (t-1)(|A|^3-|A|) $ rational equations. After multiplying by $ \det(\pi_{F}|_{L_I})\det(\pi_{F}|_{L_J}) $, we get polynomial equations, each of degree $ 2t+1 $. Let us multiply these together to an equation $ q(f_1,\ldots,f_t)\ne 0 $. 
	
	Summarized, our choice of $ (e_1,\ldots,e_{n-t}, f_1,\ldots,f_t) $ works if
	\begin{align*}
		\det(e_1,\ldots,e_{n-t},f_1,\ldots,f_t) \cdot p(e_1,\ldots,e_{n-t}) \cdot q(f_1,\ldots,f_t) \ne 0.
	\end{align*}
	This equation has degree at most $ |A|(n-t) + n + 3t^2|A|^3 $. 
	If $ S\subseteq \Kbar^n $ and $ (e_1,\ldots,e_{n-t}, f_1,\ldots,f_t) \in S^{n\times n} $ are sampled uniformly at random, then the  Schwartz-Zippel lemma, \Cref{lem:Schwartz-Zippel} shows that the probability of not satisfying the equation is at least $ 1-\dfrac{|A|(n-t) + n + 3t^2|A|^3}{|S|} $. Of course, this term is bounded from below by $ 1-\dfrac{4n^2|A|^3}{|S|} $.
\end{proof}

We need two final lemmata before we can start the proof. The contraction variety can be described by the vanishing of suitable minors of contractions $ T_f $. However, in the case of skew-symmetric matrices, it is more appropriate to describe it by the \emph{Pfaffians}. 

We remind the reader that the determinant of a skew-symmetric matrix $ M $ of even size is always the square of another polynomial in the entries of $ M $. This other polynomial is called the \emph{Pfaffian} of $ M $. More precisely, the Pfaffian $ \operatorname{pf}(M) $ of a skew-symmetric matrix $ M \in \K^{2n\times 2n}$ is the unique polynomial (up to sign) such that $ \operatorname{pf}(M)^2 = \det(M) $. If $ M \in \K^{n\times n}$ and $ I\subseteq [n] $ with $ |I| $ even, then the $ I $-\emph{sub-Pfaffian} of $ M $ is the Pfaffian of the submatrix $ M_{I, I} $ of $ M $ whose rows and columns are indexed by $ I $. We denote the $ I $-sub-Pfaffian by $ \operatorname{pf}_I(M) $ and the $ (I, I) $-minor by $ \det_{I, I}(M) $. Thus, $ \operatorname{pf}_I(M)^2 = \det_{I, I}(M) $. 

\begin{lemma}\label{lem:principal-minor-theorem}
	Let $ M\in \K^{n\times n} $ be a skew-symmetric matrix and $ m $ such that $ 1\le m < n $. Then, the following are equivalent.
	\begin{enumerate}[(a)]
		\item The rank of $ M $ is at most $ 2m $.
		\item All $ (2m+1)\times (2m+1) $ minors of $ M $ vanish.
		\item All principal $ (2m+2)\times (2m+2) $ minors of $ M $ vanish. 
		\item All principal $ (2m+2)\times (2m+2) $ sub-Pfaffians vanish. 
		\item For all $ P\in \K^{n\times (2m+2)} $, the Pfaffian of $ P^{T} M P $ vanishes. 
	\end{enumerate}
\end{lemma}
\begin{proof}
	This is the principal minor theorem for skew-symmetric matrices. See for instance \cite[Theorem 3.12]{Kodiyalam_Lam_Swan_2008} and \cite{Stoll_1952}. 
\end{proof}

\begin{lemma}\label{lem:random-subpfaffians}
	Let $ \mathcal{L}\subseteq \Lambda^2(\Kbar^n) $ be an $ N $-dimensional affine space of skew-symmetric matrices. For generic $ P_1,\ldots,P_{N+1} \in \Kbar^{n\times (2m+2)} $, the set $ X $ of matrices in $ \mathcal{L} $ of rank at most $ 2m $ equals the set of matrices $ M \in \mathcal{L}$ such that $ \operatorname{pf}(P_1^{T} M P_1) = \ldots = \operatorname{pf}(P_{N+1}^{T} M P_{N+1}) = 0 $.
	
	Furthermore, if $ S\subseteq \Kbar $ is a finite set and $ P_1,\ldots,P_{N+1} $ are sampled uniformly at random from $ S^{n\times (2m+2)} $, then the Pfaffians above cut out the set $ X $ with probability at least $ 1-\frac{4(m+1)^{N+1}}{|S|} $. 
\end{lemma}
\begin{proof}
	Let us view $ q_M := \operatorname{pf}(P^{T} M P) $ as a polynomial in $ P $ with parameter $ M $. 
	For any fixed $ M\in \mathcal{L}\setminus X $, by \Cref{lem:principal-minor-theorem} there exists $ P \in \K^{n\times (2m+2)} $ such that $  q_M(P) \ne 0 $. Hence, $ V(q_M) $ is a proper hypersurface of $ \K^{n\times (2m+2)} $ for each $ M $. Now, consider the variety 
	\begin{align*}
		Y = \{(M, P_1,\ldots,P_{N+1}) \in (\mathcal{L}\setminus X)\times (\K&^{n\times (2m+2)})^{N+1} \\
					&\mid q_M(P_1) = \ldots = q_M(P_{N+1}) = 0  \}	
	\end{align*}
	The fiber of a point $ M\in \mathcal{L}\setminus X $ has dimension at most $ (N+1)(n(2m+2)-1) $, since each $ P_i $ is constrained to a hypersurface. Hence, $ Y $ has dimension at most $ N + (N+1)(n(2m+2)-1) $. Therefore, the projection from $ Y $ to $ (\K^{n\times (2m+2)})^{N+1} $ cannot have dense image, since the target has strictly larger dimension $ (N+1)n(2m+2) $. Hence, for generic $ P_1,\ldots,P_{N+1} $, there is no matrix $ M $ of rank greater than $ 2m $ such that $ q_M(P_1) = \ldots = q_M(P_{N+1}) = 0 $. This concludes the qualitative part of the proof. 
	
	For the probabilistic statement, we loosely follow a technique outlined in \cite{lakshman1991complexity}. Write $ f_P := \operatorname{pf}(P^{T} M P) $ as a polynomial in $ M $ with coefficients in $ P $. Let $ S\subseteq \K $ be a finite set and assume that $ P_1,\ldots,P_{N+1} $ are sampled uniformly at random from $ S^{n\times (2m+2)} $. We claim that the probability of the event ``$ V(f_{P_1},\ldots,f_{P_{N+1}}) \cap \mathcal{L} \ne X$'' is at most $ \dfrac{4(m+1)^{N+1}}{|S|} $. 
	To this end, we prove by induction that for each $ j\in \{1,\ldots,N\} $, the probability that $ V(f_{P_1},\ldots,f_{P_{j}}) $ is equidimensional of codimension $ j $ is at least $ 1-2(m+1)^j $. Indeed, assume this is true for a fixed $ j $. Then, we must show that for most choices of $ P_{j+1} $, intersecting $ V(f_{P_1},\ldots,f_{P_{j}}) $  with $ V(f_{P_{j+1}}) $ makes the dimension drop in each component. By Bézout, $ V(f_{P_1},\ldots,f_{P_{j}}) $ has at most $ (m+1)^j $ irreducible components, since it is defined by $ j $ polynomials of degree $ m+1 $. For each component $ C $, pick a matrix $ M_C $ in the component. By Schwartz-Zippel, we know that the probability that $ \operatorname{pf}(P^{T} M P) = q_{M_C} $ vanishes at the random $ P = P_{j+1} $ is at most $ 2(m+1)/|S| $, since $ q_{M_C} $ has degree $ 2(m+1) $. By a union bound over all components $ C $, the probability that any of the polynomials $ q_{M_C} $ vanishes at $ P_{j+1} $ is at most $ 2(m+1)^{j+1}/|S| $. In particular, for $ j = N $ we obtain a zero-dimensional set with failure probability at most $ \sum_{j = 1}^{N} 2(m+1)^{j+1}/|S| $. For $ j = N+1 $, we see with the same argument that $ f_{P_{N+1}} $ does not vanish at any of the points of the finite set $ (V(f_{P_1},\ldots,f_{P_{N}})\cap \mathcal{L})\setminus X $ with probability at least $ 1-2(m+1)^{N+1}/|S|$ over $ P_{N+1} $. In total, we see that $ V(f_{P_1},\ldots,f_{P_{N+1}}) \cap \mathcal{L} $ fails to cut out $ X $ with probability at most $ \sum_{j = 1}^{N+1} 2(m+1)^{j+1}/|S| \le 4(m+1)^{N+1}/|S| $. The last bound used the geometric series. 
\end{proof}

\begin{theorem}\label{thm:algorithmic-alternating-kruskal}
	Let
	\begin{align}\label{eq:algorithmic-alternating-kruskal}
		T=\sum_{i=1}^r a_{i1}\wedge a_{i2}\wedge a_{i3}\in \Lambda^3(\mathbb K^n)
	\end{align}
	be an alternating tensor, and let $U_i:=\langle a_{i1},a_{i2},a_{i3}\rangle$.
	Suppose that $U_1,\ldots,U_r$ has Kruskal rank $k$ and satisfies $2r\le 3k-2$. 
	Then, \eqref{eq:algorithmic-alternating-kruskal} is the unique skew rank decomposition of $T$ and it can be recovered by
	a randomized algorithm in time $ \log(1/\varepsilon)^2 \cdot n^{\mathcal{O}(r-k+1)}$ for any desired success probability $0<\varepsilon<1$. 
\end{theorem}

\begin{proof}
	As in the proof of Theorem~\ref{thm:alternating_kruskal}, we may choose $m \in \{r-k+1,\ldots,2k-r-1\} $ and write $ h = r-m $, $ t = n-3h $. It turns out our runtime claim is achieved taking $ m = 2k-r-1 $. 
	Then the irreducible components of $X_{2m}(T)$ are the linear spaces $ U_I:=\langle U_i:i\in I\rangle^\perp $, where $ I $ ranges over all $ h $-element subsets of $ [r] $. 
	
	Pick a sufficiently large, finite set $ S\subseteq \Kbar $. Targeting a success probability of $ 1-\varepsilon/3 $, we take $ |S| = \dfrac{12n^2\binom{r}{h}^3}{\varepsilon} $. 
	Let us sample $ (e_1,\ldots,e_{3h},f_1,\ldots,f_t) \in S^{n\times n} $. Set $ E = \langle e_1,\ldots,e_{3h} \rangle $ and $ F = \langle f_1,\ldots,f_t \rangle $. Note that $ \binom{r}{h} \le r^h \le n^h = n^{\mathcal{O}(r-k)} $, using that $ h = 2r-2k+1 $.
	Sampling uniformly from $ S $ takes time $ \mathcal{O}(\log |S|) $. Hence, sampling from $ S^{n\times n} $ is feasible in time $ n^2 \mathcal{O}(\log (\frac{1}{\varepsilon}) + (1+n)\log n) $. For the last term, we used that $ \log \binom{r}{h} \le \log (n^h) \le h \log n \le n \log n$. In the following, we will do the runtime analysis for fixed $\varepsilon$. The dependency in $\varepsilon$ will be discussed later. 
	
	Next, we compute the intersections of $ X_{2m}(T) $ with the spaces of the form $ f+E $. For $ f\in F $, we write $ W_f $ for the intersection of $ X $ with $ f+E $. By \Cref{lem:principal-minor-theorem}, it holds that 
	\begin{align*}
		X_{2m}(T) \cap (f+E) = \{g\in f+E \mid \forall I\subseteq [n], |I| = 2m+2\colon \operatorname{pf}_{I}(T_g) = 0\}. 
	\end{align*}
	These intersections $ W_f $ can therefore be computed by solving the polynomial system described by the sub-Pfaffians, which are equations of degree $ m+1 $ in $ 3h $ variables on $ f+E $. With probability at least $ 1-\varepsilon/3 $ over the choice of $ E $ and $ f_1,\ldots,f_t $, the conditions of \Cref{lem:parallel-slice-component-recovery}, as stated in \Cref{lem:genericity-implied-by-schwartz-zippel}, are all satisfied. In particular,  each of the spaces $ f + E $ intersects $ X_{2m}(T) $ in a finite set, since Condition (b) of \Cref{lem:genericity-implied-by-schwartz-zippel} guarantees $ E\cap U_I = \{0\} $. Hence, all $ W_f $ are finite. 
	
	However, there are $ \binom{n}{2m+2} $ sub-Pfaffians, which is too much. Therefore, we are to solve a smaller subsystem with $ 3h $ variables and $ 3h + 1 $ equations. To this end, we sample random matrices $ P_1,\ldots,P_{3h+1} \in (S')^{n\times (2m+2)}$, where $ S'\subseteq\K $ is a finite set of appropriate size. Then, we solve the system
	\begin{align*}
		\operatorname{pf}(P_1^{T} T_{g} P_1) = \ldots = \operatorname{pf}(P_{3h+1}^{T} T_{g} P_{3h+1}) = 0, \qquad g\in f + E. 
	\end{align*} 
	By \Cref{lem:random-subpfaffians}, this new system has a finite solution set which equals $ W_f $ with probability at least $ 1-\frac{4(m+1)^{3h+1}}{|S'|} $. Targeting a failure probability of at most $ \varepsilon/3 $, we choose $ |S'| = \frac{12(2t-1)(m+1)^{3h+1}}{\varepsilon}  $. 
		
	By~\cite{lakshman1991complexity}, using Gröbner basis methods one can solve a zero-dimensional system of polynomials of degree $D$ in $N$ variables in time $D^{\mathcal{O}(N)}$ with probability at least $1-2^{-D^N}$. Specializing to our setting, we can solve this system in time $(m+1)^{\mathcal{O}(3h)}=n^{\mathcal{O}(h)} = n^{O(r-k+1)}$ with probability at least $1-2^{-(m+1)^{3h}}-\frac{1}{3}(1+\frac{1}{2t-1})\varepsilon$.  The additional $ \varepsilon $-term in the probability stems from a union bound with the estimates from before, as we need to bet that both solution sets, $ W_f $ and the zero set of the Pfaffians, are zero-dimensional. More precisely, we bet that the conditions of both \Cref{lem:parallel-slice-component-recovery}, as described in \Cref{lem:genericity-implied-by-schwartz-zippel}, and the conditions of \Cref{lem:random-subpfaffians} are fulfilled. This is the analysis for a single $ W_f $. However, note that the conditions of \Cref{lem:parallel-slice-component-recovery} guarantee that $ W_f $ is finite for all $ f $, whereas for the conditions of \Cref{lem:random-subpfaffians} we will need a union bound. 
	
	We can compute all the finite sets $ W_{f_1},\ldots,W_{f_t},W_{f_1+f_2},\ldots,W_{f_1+f_t} $ in time $ (2t-1)(m+1)^{\mathcal{O}(3h)} = n^{\mathcal{O}(h)} = n^{\mathcal{O}(r-k+1)} $. The probability that any of the zero-dimensional solves fails is bounded by $ 	 \frac{2t-1}{2^{(m+1)^{3h}}} + \frac{2}{3}\varepsilon $. The term $ \frac{2t-1}{2^{(m+1)^{3h}}} $ is already negligibly small for large $n$, but by repeating the zero-dimensional solve $\mathcal{O}(\log(1/\varepsilon))$ times, we can assume for all $n$ that it is bounded by $ \varepsilon/3 $, so that the success probability up to here is at least $ 1-\varepsilon $. Since sampling from the sets $S$ and $S'$ affects the runtime by a factor of $\mathcal{O}(\log(1/\varepsilon))$ and the zero-dimensional solve needs to be repeated potentially $\mathcal{O}(\log(1/\varepsilon))$ times, we obtain a $\varepsilon$-dependency of  $\mathcal{O}(\log(1/\varepsilon)^2)$ up to this point. As the following steps will not depend on $\varepsilon$, for brevity we omit the explicit $\varepsilon$-dependency for the remainder of the proof.
		
	By \Cref{lem:genericity-implied-by-schwartz-zippel}, the conditions of \Cref{lem:parallel-slice-component-recovery} are satisfied, and we successfully computed the spaces $ W_f $ from above with probability at least $ 1-\varepsilon $. Hence, with probability at least $ 1-\varepsilon $, we recover bases of all the subspaces $ U_I $ via the matching procedure outlined in \Cref{lem:parallel-slice-component-recovery}. 
	 
	Since the number of components is $\binom{r}{h}$, the matching step costs at most $ \mathcal{O}\left(n^2 \binom{r}{h}^3\right)=n^{\mathcal{O}(r-k+1)}$. Taking orthogonal complements gives the spaces $ \{\langle U_i:i\in I\rangle:|I|=h\} $ and has cost $n^{\mathcal{O}(r-k+1)}$ ($ \mathcal{O}(n^3) $ per vector space, but we need to it for all $ \binom{r}{m} $ spaces). Since $h\le k-1$, we can form the sets $ S_{h-1},S_{h-2},\ldots,S_1$ specified in the proof of Lemma~\ref{thm:combinatorics_general}. 
	We now estimate the cost of forming these sets. Note that $	|S_i|=\binom{r}{i}\le r^i $. 
	
	To recover $S_{i-1}$ from $S_i$, we first compute $ V\cap V' $ for all pairs $V,V'\in S_i$, and we retain $V\cap V'$ only if $\dim V\cap V'=3(i-1)$. Every true element of $S_{i-1}$ arises this way, since for $|I|=i-1$ and distinct $a,b\notin I$,
	\[
		\langle U_j:j\in I\cup\{a\}\rangle \cap	\langle U_j:j\in I\cup\{b\}\rangle = \langle U_j:j\in I\rangle.
	\]
	After sorting out duplicates of the resulting candidates, we continue to filter them: A space $V\cap V'$ is kept if and only if it is contained in $r-i+1$ members of $S_i$. By Lemma~\ref{lem:combinatorics_general}, the surviving
	subspaces are exactly the elements of $S_{i-1}$.
	To get all $ V\cap V' $ we need to compute $\mathcal{O}(|S_i|^2)$ subspace intersections, at a cost of $ \mathcal{O}(n^3) $ each for doing Gaussian elimination. Counting which members of $ S_i $ contain a pair $ V\cap V' $ takes $ |S_i| $ Gaussian eliminations, so total time at most $ \mathcal{O}(n^3|S_i|^3) $ for all pairs $ V\cap V' $. Discarding duplicates takes time at most $ \mathcal{O}(n^3|S_i|^2) $. Hence, the $ i $-th step takes time at most  $\mathcal{O}(n^3|S_i|^3) = n^{\mathcal{O}(r-k+1)} $. For the last estimate, we used that $|S_i|^3 \le r^{3i} \le n^{3h} = n^{6(r-k)+3}  $. Summing over $i=h,h-1,\ldots,2$, we stay in time $ n^{\mathcal{O}(r-k+1)} $. 

	At this point, we have recovered a basis $u_{i1},u_{i2},u_{i3}$ for each subspace $U_i$ with cost $n^{\mathcal{O}(h)}$. These bases determine the skew-rank decomposition up to scalars. In the proof of Theorem~\ref{thm:alternating_kruskal}, we saw that the 3-tensors $ (u_{i1}\wedge u_{i2}\wedge u_{i3})_{i=1,\ldots,r} $ are linearly independent. Hence the coefficients $\lambda_i$ in
	\[
		T=\sum_{i=1}^r \lambda_i\tau_i
	\]
	are obtained by solving a linear system, which takes time $\mathcal{O}(r^2 n^3)$. Combining the above steps yields a randomized algorithm with running time $ n^{\mathcal{O}(h)} = n^{\mathcal{O}(r-k+1)}$ and success probability $ 1-\varepsilon $. This concludes the proof. 
\end{proof}

\begin{algorithm}[t]
	\caption{\textsc{Alternating Kruskal Decomposition}}
	\label{alg:alternating-kruskal}
	\small
	\algrenewcommand\algorithmicensure{\textbf{Output:}}
	\begin{algorithmic}[1]
		\Require An alternating tensor $T\in\Lambda^3(\K^n)$, integers $r,k$ with
		$2r\le 3k-2$, and a failure tolerance $0<\varepsilon<1$.
		The input is promised to admit
		$T=\sum_{i=1}^r a_{i1}\wedge a_{i2}\wedge a_{i3}$, where the associated
		$3$-spaces $U_i=\langle a_{i1},a_{i2},a_{i3}\rangle$ have Kruskal rank
		at least $k$.
		\Ensure The unique minimum skew-rank decomposition of $T$.
		
		\State Set $m\gets 2k-r-1$, $h\gets r-m=2r-2k+1$, and
		$t\gets n-3h$.
		\State Over $\Kbar$, construct the spaces
		$E\gets\langle e_1,\ldots,e_{3h}\rangle$ and
		$F\gets\langle f_1,\ldots,f_t\rangle$ from independently sampled
		random vectors $e_1,\ldots,e_{3h},f_1,\ldots,f_t$. 
		\State Independently choose random matrices
		$P_1,\ldots,P_{3h+1}\in\Kbar^{n\times(2m+2)}$.
		Choose the sampling ranges as in the proof, so that the total failure
		probability is at most $\varepsilon$.
		\State Set
		$\mathcal F\gets\{f_1,\ldots,f_t\}
		\cup\{f_1+f_i:2\le i\le t\}$.
		
		\ForAll{$f\in\mathcal F$}
			\State Solve the $3h+1$ Pfaffian equations
			$\operatorname{pf}(P_j^TT_gP_j)=0$ on $g\in f+E$ 
			\State to compute $W_f\gets X_{2m}(T)\cap(f+E)$.
		\EndFor
		
		\State For every $x\in W_{f_1}$ and $i=2,\ldots,t$, match $x$ with
		the unique $x_i\in W_{f_i}$ such that
		$x+x_i\in W_{f_1+f_i}$.
		\State Set $\mathcal L\gets	\{\langle x,x_2,\ldots,x_t\rangle:x\in W_{f_1}\}$. 
		\Comment{Each matched $ t $-tuple spans one component of $ X_{2m}(T) $. }
		
		\State Compute 	$\mathcal S_h\gets\{L^\perp:L\in\mathcal L\}$. \hspace*{-2em}\Comment{nb: $\mathcal{S}_h = \{\langle U_j:j\in I\rangle:I\subseteq[r],\ |I|=h\}$.}
		\For{$i=h,h-1,\ldots,2$}
			\State Calculate bases for the distinct pairwise intersections of members of
			$\mathcal S_i$.
			\State Keep those intersections which have dimension $3(t-1)$ and lie in at least 
			\State $r-i+1$ members of $\mathcal S_i$.
			\State Call the resulting family $\mathcal S_{i-1}$.
		\EndFor
		\State We computed bases for each element of 
		$\mathcal S_1=\{U_1,\ldots,U_r\}$.
		
		\For{$i=1,\ldots,r$}
			\State Choose a basis $(u_{i1},u_{i2},u_{i3})$ of $U_i$ and set
			$M_i \gets u_{i1}\wedge u_{i2}\wedge u_{i3}$.
		\EndFor
		\State Solve the linear system
		$T=\sum_{i=1}^r\lambda_i M_i$ to obtain unique $ \lambda_1,\ldots,\lambda_r\in \K $. 
		\State \Return $T=\sum_{i=1}^r u_{i1}\wedge u_{i2}\wedge(\lambda_i u_{i3})$.
		\State If any of the steps above cannot be executed, output ``fail''. \\
		\Comment{The algorithm can always be repeated after resampling the random choices. }
	\end{algorithmic}
\end{algorithm}

\section{An intrinsic definition of Kruskal tensors}\label{sec:algorithms-and-examples}
Initially, we defined a Kruskal tensor as a tensor which has a decomposition that satisfies Kruskal's condition. However, being a Kruskal tensor does not depend on the choice of decomposition, since there is only one minimum rank decomposition. It is therefore desirable to have an intrinsic definition of Kruskal tensors. An intrinsic definition should use an open condition on the tensor itself, rather than its decomposition. We will give an intrinsic definition using the contraction variety. 
As the first step, we are going to define the intrinsic k-ranks $ k_1(T), k_2(T) $ and $ k_3(T) $ of a tensor. 

\begin{defi}\label{def:intrinsic-Kruskal}
	Let $ T\in \K^{n_1\times n_2\times n_3} $ be a tensor of rank $ r $. 
	\begin{enumerate}[(a)]
		\item For $ k\in \{0,\ldots,\min\{n_i, r\}\} $, we say that $ T $ has the property $ K_{i,k} $, if the dimension of $ X_{r-k}^{(i)}(T) $ is at most $ n_i-k $.
		\item The \emph{intrinsic $ k_i $-rank} of $ T $, denoted $ k_i(T) $, is defined as the largest value $ k\in \{0,\ldots,\min\{n_i, r\} \} $ such that $ T $ has the property $ K_{i, k} $. 
	\end{enumerate}
\end{defi}

Since $ X_r^{(i)}(T) = \Kbar^{n_i} $, any tensor has the property $ K_{i, 0} $. Hence, the intrinsic $ k_i $-rank is a well-defined number in $ \{0,\ldots,\min\{n_i, r\}\} $. We will show in the following proposition that under Kruskal's inequality, the intrinsic k-ranks and the k-ranks of the minimum rank decomposition coincide. Without Kruskal's inequality, one obtains a lower bound. As an important caveat, we note that for tensors of high ranks much beyond the range of Kruskal's theorem, the lower bound given by the intrinsic Kruskal ranks will most often just be zero. E.g., if $ r\ge 2n_1 $ and $ T \in \K^{n_1\times n_2\times n_3} $ is a generic rank $ r $ tensor with $ n_1\ge n_2\ge n_3 $, then $ k_1(T) = 0 $, since $ r-k\ge n_1 \ge \min\{n_2, n_3\} $ for all $ k $ in the range of \Cref{def:intrinsic-Kruskal}(a). 
Furthermore, if a tensor $ T $ has rank $ r $ and satisfies the property $ K_{i, k} $, then it need not satisfy the property $ K_{i,k-1} $. An example will be shown in the appendix. 
Hence, it is not possible to obtain lower bounds for the intrinsic k-ranks by checking the property only for smaller values of $ k $. This is in contrast with the standard notion of k-rank.

\begin{prop}\label{prop:intrinsic-krank-recovers-krank}
	Assume that $ T \in \K^{n_1\times n_2\times n_3} $ is a tensor of rank $ r $. Let $ T = \sum_{i = 1}^{r} a_i\otimes b_i\otimes c_i $ be a tensor decomposition with $ k $-ranks $ k_a, k_b $ and $ k_c $. Then, $ k_a\ge k_{1}(T), k_b \ge k_2(T)$ and $k_c\ge k_3(T)$. Furthermore, if $ 2r\le k_1(T) + k_2(T) + k_3(T) - 1$, then equality holds between the k-ranks of the decomposition and the intrinsic k-ranks of $ T $. 
\end{prop}
\begin{proof}
	Let $k=k_1(T)$. The obvious inclusion 
	\begin{align*}
		 X_{r-k}^{(1)}(T) \supseteq \bigcup_{i_1 < \ldots < i_{k}} \langle a_{i_1},\ldots,a_{i_{k}} \rangle^{\perp}
	\end{align*}
	shows that the dimension of each space $ \langle a_{i_1},\ldots,a_{i_{k}} \rangle^{\perp} $ is at most the dimension of $ X_{r-k}^{(1)}(T) $, which by assumption is $ n_1-k $. Therefore, $ a_{i_1},\ldots,a_{i_{k}} $ are linearly independent for all $i_1 < \ldots < i_{k}$. Hence, the Kruskal rank $ k_a $ is at least $ k $. We conclude that $ k_a\ge k_1(T) $. If $ 2r\le k_a+k_b+k_c-2 $, then our proof of Kruskal's uniqueness theorem shows equality between $ k_a $ and $ k_1(T) $. 
	
	Now, assume that $ 2r\le k_1(T) + k_2(T) + k_3(T) - 1$. If $ k_a > k_1(T) $, then we would have $ 2r\le k_a+k_b+k_c-2 $ and our proof of Kruskal's uniqueness theorem would still show equality between $ k_a $ and $ k_1(T) $, which contradicts the assumption $ k_a > k_1(T) $. Hence, $ k_a = k_1(T) $. The arguments for $ k_2(T) $ and $ k_3(T) $ are analogous. 
\end{proof}

\noindent
Finally, we state an intrinsic definition of Kruskal tensors. 
\begin{cor}[{Intrinsic definition of Kruskal tensors}]\label{cor:intrinsic-definition-kruskal}
	A tensor $ T \in \K^{n_1\times n_2\times n_3} $ of rank $ r $ is a Kruskal tensor if and only if $ 2r\le k_1(T) + k_2(T) + k_3(T) - 2 $. 
\end{cor}

\begin{cor}\label{cor:Kruskal-tensors-open}
	The set of Kruskal tensors of rank $ r $ is a Zariski open subset of the set of tensors of rank $ r $. 
\end{cor}
\begin{proof}
	Let us denote by $ \mathcal{R}_{=r} $ the constructible subset of $ \sigma_{r} $ of tensors whose rank is exactly $ r $. We consider Kruskal's intrinsic inequality, 
	\begin{align*}
		2r\le k_1(T) + k_2(T) + k_3(T)-2.
	\end{align*} 
	Define the sets $ \mathcal{U}_{p,q,\ell} := \{T \in \mathcal{R}_{=r} \mid k_1(T)\ge p, k_2(T)\ge q, k_3(T)\ge \ell  \}$. We claim that the sets $ \mathcal{U}_{p,q,\ell} $ are Zariski open subsets of $ \mathcal{R}_{=r} $. Indeed, by definition of the intrinsic k-ranks, a tensor lies in $ \mathcal{U}_{p, q, \ell} $ if and only if there exists some $ k\ge p $ such that the dimension of $ X_{r-k}^{(1)}(T) $ is at most $ n_1-k $ (plus analogous conditions for the other two indices). The fiber semicontinuity theorem (\cite[Exercise 3.22, Chapter II]{hartshorne2013algebraic}) yields that these conditions are Zariski open, and so are thus their unions and finite intersections. 
	The set of tensors in $ \mathcal{R}_{=r} $ satisfying Kruskal's inequality can be written as a finite union
	\begin{align*}
		\{T\in \mathcal{R}_{=r} \mid T \text{ satisfies Kruskal's inequality} \} = \bigcup_{\substack{(p, q, \ell)\\ 2r+2= p+q+\ell} } \mathcal{U}_{p, q, \ell} \cap \mathcal{R}_{=r}
	\end{align*}
	Hence, the set of tensors in $ \mathcal{R}_{=r} $ satisfying Kruskal's inequality is an open subset of $ \mathcal{R}_{=r} $.
\end{proof}

\begin{cor}
	Deciding whether the dimension of a contraction variety is at most $ d $ is NP-hard. 
\end{cor}
\begin{proof}
	It is well-known to be NP hard to decide, given $ a_1,\ldots,a_r \in \Q^n $ and $ d $ whether the k-rank of a given list of vector $ a_1,\ldots,a_r $ satisfies $ k_a\le d $. See \cite{Khachiyan_1995} and \cite{Tillmann_Pfetsch_2014}. From \cite[Proof of Theorem 1]{Tillmann_Pfetsch_2014}, it follows that it is still NP hard in the worst case to decide whether $ k_a=d $ if one is given the promise that $ k_a\ge d $. 
	Given an algorithm $ \mathcal{A} $ to decide whether $ \codim X_{m}(T) \le d' $ on input $ (T, m, d') $, we construct an algorithm $ \mathcal{B} $ to decide whether $ k_a = d $ on input $ (a, d) $ with promise $ k_a\ge d $ which arises from a polynomial reduction.
	Algorithm $ \mathcal{B} $ works as follows. Given $ a_1,\ldots,a_r $ and $ d $, take $ b_i = e_i \in \Q^r$ and choose $ c_1,\ldots,c_r \in \Q^{s}$ as distinct points on the rational normal curve, so that $ c_i = (1, \lambda_i,\ldots,\lambda_i^{s-1}) $ for distinct $ \lambda_i\in \Q $ and $ s = r-d+1 $. Compute the tensor $ T = \sum_{i = 1}^{r} a_i\otimes b_i \otimes c_i $. We have $ k_c = s, k_b = r $, and thus $ k_a+k_b+k_c \ge d + r + s = 2r+1$. Hence, Kruskal's inequality is satisfied if and only if $ k_a > d $. However, as Kruskal's inequality is violated by at most 1, \Cref{prop:intrinsic-krank-recovers-krank} still applies and thus the intrinsic k-ranks of $ T = \sum_{i = 1}^{r} a_i\otimes b_i \otimes c_i $ equal the respective k-ranks of $ a,b $ and $ c $. Also, \Cref{thm:Kruskal_border_rank_precise} shows that $ \rk (T) = r $. Hence, $ k_a > d $ if and only if $ \codim X_{r-d-1}(T) = d+1$. Thus, we call algorithm $ \mathcal{A} $ with input $ (T, r-d-1, d) $ and we let $ \mathcal{B} $ output ``yes'' if the output of $ \mathcal{A} $ is ``yes''. 
\end{proof}

\begin{remark}
	The example given in the proof above was chosen such that if $ k_a = d+1 $ then it is not possible to choose any other $ m $ to make the uniqueness proof from \Cref{thm:kruskal} work. 
	This indicates that any algorithm which implicitly reveals the dimensions of contraction varieties (like Domanov-Lathauwer \cite{domanov2014canonical} and \cite{BSS26} and our algorithm from \Cref{sec:algorithmic_alternating} do) cannot run in polynomial time unless P = NP. 	
\end{remark}

\begin{remark}
	A similar definition of intrinsic alternating Kruskal tensors can be given. This is not demonstrated, since it is completely analogous. 
\end{remark}

	\subsection*{Acknowledgments}
    The authors acknowledge helpful and inspiring discussions with Fulvio Gesmundo, Tim Seynnaeve, Daniele Taufer, Nick Vannieuwenhoven and Aravindan Vijayaraghavan. 
    
    The authors used GPT-5.6 Sol for editorial polishing, literature review, to generate the example at the end of Section~\ref{sec:standard-kruskal}, and to generate ideas for the counterexample in Theorem~\ref{thm:alternating_sharp}. The example in Appendix A was generated with the help of hand-written Julia code~\cite{Julia-2017}. The authors take full responsibility for the correctness, exposition, and attribution in the manuscript. 
    
    Benjamin Lovitz acknowledges support from NSERC Discovery Grant RGPIN-2026-05413. Alexander Taveira Blomenhofer is supported by the Quantum for Life center, funded by the Novo Nordisk Foundation (grant NNF20OC0059939 ‘Quantum for Life’). Furthermore, he also acknowledges financial support from the VILLUM FONDEN via the QMATH Centre of Excellence (Grant No. 10059). 

	\bibliography{bibML}
	\bibliographystyle{plain}
	
	\appendix

\section{Examples}
In this section, we examine a concrete symmetric tensor and we check whether it is a Kruskal tensor. We calculate its contraction varieties and its minimum rank decomposition. 
We will see that for $ r = 5 $, it satisfies the properties $ K_{i, 0}, K_{i, 3} $ and $ K_{i, 4} $ from \Cref{def:intrinsic-Kruskal} for $ i = 1,2,3 $, but that it does not satisfy $ K_{i, 1} $ and $ K_{i, 2} $. 

Consider the symmetric tensor $T \in \Q^{4\times 4\times 4}$ given by
\begin{align*}
	T = &e_1^3 + 3e_1^2e_2 + 9e_1^2e_3 + 3e_1e_2^2 + 6e_1e_2e_3 + 6e_1e_2e_4 + 3e_1e_3^2 + 6e_1e_4^2 + 2e_2^3\\ + &6e_2^2e_3 + 6e_2^2e_4 + 6e_2e_3^2 + 12e_2e_3e_4 + 6e_2e_4^2 + 4e_3^3 + 3e_3^2e_4 + 9e_3e_4^2  
\end{align*}
Here, $a_1a_2a_3$ stands for the symmetrization of the tensor $a_1 \otimes a_2 \otimes a_3$, defined by $ a_1a_2a_3 =  \frac{1}{3!} \sum_{\sigma\in \mathfrak{S}_3} a_{\sigma(1)}\otimes a_{\sigma(2)}\otimes a_{\sigma(3)} $. The slices of $T$ along the third index are given by  
\[
T_{e_1}=
\begin{pmatrix}
	1 & 1 & 3 & 0 \\
	1 & 1 & 1 & 1 \\
	3 & 1 & 1 & 0 \\
	0 & 1 & 0 & 2
\end{pmatrix},
\quad
T_{e_2}=
\begin{pmatrix}
	1 & 1 & 1 & 1 \\
	1 & 2 & 2 & 2 \\
	1 & 2 & 2 & 2 \\
	1 & 2 & 2 & 2
\end{pmatrix},
\]
\[
T_{e_3}=
\begin{pmatrix}
	3 & 1 & 1 & 0 \\
	1 & 2 & 2 & 2 \\
	1 & 2 & 4 & 1 \\
	0 & 2 & 1 & 3
\end{pmatrix},
\quad
T_{e_4}=
\begin{pmatrix}
	0 & 1 & 0 & 2 \\
	1 & 2 & 2 & 2 \\
	0 & 2 & 1 & 3 \\
	2 & 2 & 3 & 0
\end{pmatrix}.
\]
Given the representation of $T$ above, it is possible to compute the contraction varieties using computer algebra systems. We see that the rank of $T$ must be at least $4$, since its contraction $T_{e_3}$ has rank $4$. More precisely, the slices $T_{e_1},T_{e_2},T_{e_3},T_{e_4}$ have ranks $3,2,4,4$, respectively. Since $T$ is symmetric, there is only one (intrinsic) k-rank, $k_1(T)$. Note that the generic rank of symmetric $4\times 4\times 4$ tensors is $5$, and $5 = 1.5\cdot 4-1$ is the largest rank for which there exist $4\times 4\times 4$ Kruskal tensors. We therefore have to check whether $ T $ is a rank-$ r $ Kruskal tensor for $ r \in \{4, 5\} $. 

As mentioned in \cite{Blomenhofer_Lovitz_2025}, contraction varieties are easy to compute in constant codimension by intersecting with a plane of constant dimension. We first compute the contraction variety $X_3(T)$. The defining polynomial is the quartic 
\[
\begin{aligned}
	\det(T_f) = &-f_1^3 f_2 - f_1^3 f_3 + 3f_1^3 f_4 - f_1^2 f_2^2 - 2f_1^2 f_2 f_3 + 7f_1^2 f_2 f_4 \\
	&- f_1^2 f_3^2 + 7f_1^2 f_3 f_4 + 4f_1^2 f_4^2 + f_1 f_2^2 f_4 + f_1 f_2 f_3^2 - f_1 f_2 f_3 f_4 \\
	&- 5f_1 f_2 f_4^2 + f_1 f_3^3 - 5f_1 f_3^2 f_4 - 4f_1 f_3 f_4^2 - 6f_1 f_4^3 \\
	&+ f_2^2 f_3^2 - 3f_2^2 f_3 f_4 + f_2^2 f_4^2 + 2f_2 f_3^3 - 12f_2 f_3^2 f_4 + 4f_2 f_3 f_4^2 \\
	&+ 2f_2 f_4^3 + f_3^4 - 9f_3^3 f_4 - 4f_3^2 f_4^2 + 7f_3 f_4^3 + f_4^4.
\end{aligned}
\]
It is possible to compute that $\det(T_f)$ is an irreducible polynomial in $f$. In particular, this means that $T$ cannot be a Kruskal tensor of rank $4$ or less. Indeed, if $r = 4$, then $X_3(T)$ would be a union of $4$ hyperplanes by \Cref{lem:contraction-variety-dimension}. If $r < 4$, then $X_3(T) = \C^n$. Hence $ T $ can only be a Kruskal tensor for $r=5$. To check whether $T$ is a Kruskal tensor of rank $5$, we can compute $X_1(T)$. A computation reveals that this contraction variety is of dimension $ 0 = n-4$, and thus $k_1(T) = 4$ and Kruskal's inequality $2r + 2 \le 3k_1(T)$ holds for $ r = 5 $, if $ \rk(T) = 5 $. However, it remains to show that $ T $ has rank at most $ 5 $. To this end, we will reveal a decomposition by following the proofs of \Cref{thm:kruskal}, \Cref{thm:combinatorics_general} and \Cref{lem:combinatorics_general}, which reconstruct the decomposition from the irreducible components of $ X_{2}(T) $ (note that $ 3 = k_1(T)-1 = r-m $, where $ m = 2 $).

By intersecting with a plane, we calculate that $X_2(T)$ is an arrangement of $10 = \binom{5}{3}$ distinct lines, given by the projective points 
\begin{align*}
	&[0: 2: -1: -1],
	[0: 1: -1: 0],
	[0: 1: 0: -1],
	[1: -3: 1: 2],
	[0: 1: 0: 0],\\
	&[1: -4: 1: 2],
	[1: 1: -1: 0],
	[1: 0: -1: 0],
	[1: -1: 1: 0],
	[1: -2: 1: 0]    
\end{align*}
We have to find vectors $ a_1,\ldots,a_5 $ such that each line above can be written as $ \langle a_{i_1},a_{i_2}, a_{i_3} \rangle^{\perp} $ for some $ i_1,i_2,i_3\in [5] $. We consider the orthogonal complements of the lines, which are spaces of dimension $ 3 $, and we enumerate them by $ U_1 = \ell_1^{\perp},\ldots,U_{10} = \ell_{10}^{\perp} $ in the order above. Among the $ 120 = \binom{10}{3} $ triples of spaces $ \{U_i, U_j, U_k\} $, there are exactly 10 for which the intersection $ U_i\cap U_j\cap U_k $ has dimension $ 2 $. These are: 
\begin{align*}
	&(1,2,3), (1,4,7), (1,6,8), (2,7,9), (2,8,10),\\ &(3,4,9), (3,6,10), (4,5,6), (5,7,8), (5,9,10).
\end{align*}
We find that 
\[
\begin{array}{c|c}
	\text{Intersection} & \text{Basis } \langle b_1,b_2\rangle \\[2pt]
	\hline
	U_1\cap U_2\cap U_3
	& (1,0,0,0),\,(0,1,1,1) \\[2pt]
	\hline
	U_1\cap U_4\cap U_7
	& (1,1,2,0),\,(-1,1,0,2) \\[2pt]
	\hline
	U_1\cap U_6\cap U_8
	& (2,1,2,0),\,(0,1,0,2) \\[2pt]
	\hline
	U_2\cap U_7\cap U_9
	& (0,1,1,0),\,(0,0,0,1) \\[2pt]
	\hline
	U_2\cap U_8\cap U_{10}
	& (1,1,1,0),\,(0,0,0,1) \\[2pt]
	\hline
	U_3\cap U_4\cap U_9
	& (-1,0,1,0),\,(1,1,0,1) \\[2pt]
	\hline
	U_3\cap U_6\cap U_{10}
	& (-1,0,1,0),\,(2,1,0,1) \\[2pt]
	\hline
	U_4\cap U_5\cap U_6
	& (-1,0,1,0),\,(-2,0,0,1) \\[2pt]
	\hline
	U_5\cap U_7\cap U_8
	& (1,0,1,0),\,(0,0,0,1) \\[2pt]
	\hline
	U_5\cap U_9\cap U_{10}
	& (-1,0,1,0),\,(0,0,0,1) \\[2pt]
\end{array}
\]
Note that in the notation of \Cref{lem:combinatorics_general}, we have $ r = 5 $, $ \ell = 3 $ and $ r-\ell+1 = 3 $, which is why we have to consider intersections of triplets $ U_i\cap U_j\cap U_k $. \Cref{lem:combinatorics_general} guarantees that the 10 intersections above are of the form $ \langle a_{i_1}, a_{i_2} \rangle $ for some $ i_1 < i_2 $, if $ T $ is a Kruskal tensor. Applying \Cref{lem:combinatorics_general} now again with $ \ell = 2 $ and $ r-\ell+1 = 4 $, we see that there exist exactly $ 5 $ quadruplets of the 10 triplets above whose intersection has dimension one. Writing $ U_{ijk} $ as a shorthand for $ U_i\cap U_j\cap U_k $, we find \bigskip

\begin{tabular}{c|c|c}
	Intersection & Line & Point\\
	\hline
	$U_{3,4,9} \cap U_{3,6,10} \cap U_{4,5,6} \cap U_{5,9,10}$
	& $[1 : 0 : -1 : 0]$ & $ a_1 $\\
	\hline
	$U_{1,2,3} \cap U_{1,6,8} \cap U_{2,8,10} \cap U_{3,6,10}$
	& $[1 : 1 : 1 : 1]$ & $ a_2 $\\
	\hline
	$U_{1,2,3} \cap U_{1,4,7} \cap U_{2,7,9} \cap U_{3,4,9}$
	& $[0 : 1 : 1 : 1]$ & $ a_3 $\\
	\hline
	$U_{2,7,9} \cap U_{2,8,10} \cap U_{5,7,8} \cap U_{5,9,10}$
	& $[0 : 0 : 0 : -1]$ & $ a_4 $\\
	\hline
	$U_{1,4,7} \cap U_{1,6,8} \cap U_{5,7,8} \cap U_{4,5,6}$
	& $[1 : 0 : 1 : -1]$ & $ a_5 $\medskip \\
\end{tabular}

The five points $ a_1,\ldots,a_5 $ on the right hand side of the table correspond to the rank decomposition of $ T $. In general, this is only true up to scalars, but the missing scalars can be corrected by solving a linear system. In our case, it turns out that 
\begin{align*}
	T = \begin{pmatrix} -1\\ 0\\1\\0 \end{pmatrix}^{\otimes 3} 
	+ \begin{pmatrix} 1\\ 1\\1\\1 \end{pmatrix}^{\otimes 3} 
	+ \begin{pmatrix} 0\\ 1\\1\\1 \end{pmatrix}^{\otimes 3}
	+ \begin{pmatrix} 0\\ 0\\0\\-1 \end{pmatrix}^{\otimes 3} 
	+ \begin{pmatrix} 1\\ 0\\1\\-1 \end{pmatrix}^{\otimes 3}.
\end{align*}
We obtain the following decomposition of $ X_2(T) $, in terms of the tensor rank decomposition of $ T $.
\[
\begin{array}{c|c}
	\text{line} &
	\text{irreducible component of $ X_2(T) $} \\ \hline
	[0:2:-1:-1] &
	\langle a_2,a_3,a_5\rangle^\perp \\[2mm]
	[0:1:-1:0] &
	\langle a_2,a_3,a_4\rangle^\perp \\[2mm]
	[0:1:0:-1] &
	\langle a_1,a_2,a_3\rangle^\perp \\[2mm]
	[1:-3:1:2] &
	\langle a_1,a_3,a_5\rangle^\perp \\[2mm]
	[0:1:0:0] &
	\langle a_1,a_4,a_5\rangle^\perp \\[2mm]
	[1:-4:1:2] &
	\langle a_1,a_2,a_5\rangle^\perp \\[2mm]
	[1:1:-1:0] &
	\langle a_3,a_4,a_5\rangle^\perp \\[2mm]
	[1:0:-1:0] &
	\langle a_2,a_4,a_5\rangle^\perp \\[2mm]
	[1:-1:1:0] &
	\langle a_1,a_3,a_4\rangle^\perp \\[2mm]
	[1:-2:1:0] &
	\langle a_1,a_2,a_4\rangle^\perp
\end{array}
\]

The method above can be seen as a geometric re-interpretation of the algorithms in \cite{domanov2014canonical} and \cite{BSS26}.

\end{document}